\documentclass[MSNbibl,number,citesort,seceqn,dvips]{arxbj}

\aid{0}
\volume{20}
\issue{4}
\pubyear{2014}
\firstpage{1765}
\lastpage{1801}
\doi{10.3150/13-BEJ540} %kopijuoti is PTS

\makeatletter
\newcommand{\RMe}{\mathrm{e}}

\newcommand{\mrmd}{\,\mathrm{d}}
\newcommand{\mrmdd}{\mathrm{d}}

\newcommand{\rrvert}{\vert}
\newcommand{\llvert}{\vert}

\newtheorem{thrm}{Theorem}[section]
\newtheorem{prop}{Proposition}[section]
\newtheorem{lemma}{Lemma}[section]
\newtheorem{crll}[thrm]{Corollary}

\newproclaim{defn}{Definition}[section]

\newremark{rem}{Remark}[section]
\newremark{exmpl}{Example}[section]

\newcommand{\cl}{\operatorname{cl}}
\newcommand{\Id}{\operatorname{Id}}
\newcommand{\diag}{\operatorname{diag}}
\newcommand{\Prb}{\mathbf{P}}
\newcommand{\Mean}{\mathbf{E}}
\newcommand{\Law}{\operatorname{Law}}
\newcommand{\range}{\operatorname{range}}
\renewcommand{\phi}{\varphi}
\renewcommand{\epsilon}{\varepsilon}

\makeatother

\begin{document}
\begin{frontmatter}

\title{On the form of the large deviation rate function for the
empirical measures of weakly interacting systems}
\runtitle{Form of the rate function for weakly interacting systems}

\begin{aug}
%%%% inicialai - be tarpu
\author{\inits{M.}\fnms{Markus} \snm{Fischer}\corref{}\ead[label=e1]{fischer@math.unipd.it}}% \and
%%\runauthor{} %% auto
\address{Department of Mathematics,
University of Padua, via Trieste 63, 35121 Padova, Italy.\\
\printead{e1}}
\end{aug}

% HISTORY:
\received{\smonth{8} \syear{2012}}
\revised{\smonth{3} \syear{2013}}

% ABSTRACT
%
\begin{abstract}
A basic result of large deviations theory is Sanov's theorem, which
states that the sequence of empirical measures of independent and
identically distributed samples satisfies the large deviation principle
with rate function given by relative entropy with respect to the common
distribution. Large deviation principles for the empirical measures are
also known to hold for broad classes of weakly interacting systems.
When the interaction through the empirical measure corresponds to an
absolutely continuous change of measure, the rate function can be
expressed as relative entropy of a distribution with respect to the law
of the McKean--Vlasov limit with measure-variable frozen at that
distribution. We discuss situations, beyond that of tilted
distributions, in which a large deviation principle holds with rate
function in relative entropy form.
\end{abstract}

% KEYWORDS
% visi is mazosios raides ir pagal abecele
%

\end{frontmatter}

%s1 #&#
\section{Introduction} \label{SectIntro}

Weakly interacting systems are families of particle systems whose
components, for each fixed number $N$ of particles, are statistically
indistinguishable and interact only through the empirical measure of
the $N$-particle system. The study of weakly interacting systems
originates in statistical mechanics and kinetic theory; in this
context, they are often referred to as mean field systems.

The joint law of the random variables describing the states of the
$N$-particle system of a weakly interacting system is invariant under
permutations of components, hence determined by the distribution of the
associated empirical measure. For large classes of weakly interacting
systems, the law of large numbers is known to hold, that is, the
sequence of $N$-particle empirical measures converges to a
deterministic probability measure as $N$ tends to infinity. The limit
measure can often be characterized in terms of a limit equation, which,
by extrapolation from the important case of Markovian systems, is
called McKean--Vlasov equation (cf. McKean \cite{mckean66}). As with the
classical law of large numbers, different kinds of deviations of the
prelimit quantities (the $N$-particle empirical measures) from the
limit quantity (the McKean--Vlasov distribution) can be studied. Here we
are interested in large deviations.

Large deviations for the empirical measures of weakly interacting
systems, especially Markovian systems, have been the object of a number
of works. The large deviation principle is usually obtained by
transferring Sanov's theorem, which gives the large deviation principle
for the empirical measures of independent and identically distributed
samples, through an absolutely continuous change of measure. This
approach works when the effect of the interaction through the empirical
measure corresponds to a change of measure which is absolutely
continuous with respect to some fixed reference distribution of product
form. Sanov's theorem can then be transferred using Varadhan's lemma.
In the case of Markovian dynamics, such a change-of-measure argument
yields the large deviation principle on path space; see
L{\'e}onard \cite{leonard95} for non-degenerate jump diffusions,
Dai Pra and den Hollander \cite{daipradenhollander96} for a model of
Brownian particles in a
potential field and random environment, and Del Moral and Guionnet
\cite{delmoralguionnet98}
for a class of discrete-time Markov processes. An extension of
Varadhan's lemma tailored to the change of measure needed for empirical
measures is given in Del Moral and Zajic \cite{delmoralzajic03} and
applied to a variety
of non-degenerate weakly interacting systems. The large deviation rate
function in all those cases can be written in relative entropy form,
that is, expressed as relative entropy of a distribution with respect
to the law of the McKean--Vlasov limit with measure-variable frozen at
that distribution; cf. Remark \ref{RemFormRateFnct} below.

In the case of Markovian dynamics, the large deviation principle on
path space can be taken as the first step in deriving the large
deviation principle for the empirical processes; cf. L{\'e}onard \cite
{leonard95} or
Feng \cite{feng94,feng94b}. In Dawson and G{\"a}rtner \cite
{dawsongaertner87}, the large deviation principle for the empirical
processes of weakly interacting It{\^o} diffusions with non-degenerate
and measure-independent diffusion matrix is established in
Freidlin--Wentzell form starting from a process level representation of
the rate function for non-interacting It{\^o} diffusions. The large
deviation principle for interacting diffusions is then derived by time
discretization, local freezing of the measure variable and an
absolutely continuous change of measure with respect to the resulting
product distributions. A similar strategy is applied in Djehiche and
Kaj \cite{djehichekaj95} to a class of pure jump processes.

A different approach is taken in the early work of Tanaka \cite
{tanaka82}, where the contraction principle is employed to derive the
large deviation principle on path space for the special case of It{\^o}
diffusions with identity diffusion matrix. The contraction mapping in
this case is actually a bijection. Using the invariance of relative
entropy under bi-measurable bijections, the rate function is shown to
be of relative entropy form. In L{\'e}onard \cite{leonard95b}, the
large deviation upper bound, not the full principle, is derived by
variational methods using Laplace functionals for certain pure jump
Markov processes that do not allow for an absolutely continuous change
of measure. In Budhiraja, Dupuis and Fischer \cite{budhirajaetal12},
the path space Laplace principle for weakly interacting It{\^o}
processes with measure-dependent and possibly degenerate diffusion
matrix is established based on a variational representation of Laplace
functionals, weak convergence methods and ideas from stochastic optimal
control. The rate function is given in variational form.

The aim of this paper is to show that the large deviation principle
holds with rate function in relative entropy form also for weakly
interacting systems that do not allow for an absolutely continuous
change of measure with respect to product distributions. The large
deviation principle in that form is a natural generalization of Sanov's
theorem. Two classes of systems will be discussed: noise-based systems
to which the contraction principle is applicable, and systems described
by weakly interacting It{\^o} processes.

%re1.1 #&#
\begin{rem} \label{RemMeasureTopology}
The random variables representing the states of the particles will be
assumed to take values in a Polish space. The space of probability
measures over a Polish space will be equipped, for simplicity, with the
standard topology of weak convergence. Continuity of a functional with
respect to the topology of weak convergence might be a rather
restrictive condition. This restriction can be alleviated by
considering the space of probability measures that satisfy an
integrability condition (e.g., finite moments of a certain
order), equipped with the topology of weak(-star) convergence with
respect to the corresponding class of continuous functions
(for instance, Section 2b) in L{\'e}onard \cite{leonard95}). The results
presented below can
be adapted to this more general situation.
\end{rem}

The rest of this paper is organized as follows. In Section \ref
{SectBasics}, we collect basic definitions and results of the theory of
large deviations in the context of Polish spaces that will be used in
the sequel; standard references for our purposes are Dembo and Zeitouni
\cite{dembozeitouni98} and Dupuis and Ellis~\cite{dupuisellis97}. In
Section \ref{SectRightForm}, we introduce a toy model of discrete-time
weakly interacting systems to illustrate the use of Varadhan's lemma,
which in turn yields, at least formally, a representation of the rate
function in relative entropy form. In Section \ref{SectNoiseBased}, a
class of weakly interacting systems is presented to which the
contraction principle is applicable but not necessarily the usual
change-of-measure technique. The large deviation rate function is shown
to be of the desired form thanks to a contraction property of relative
entropy. In Section \ref{SectIto}, we discuss the case of weakly
interacting It{\^o} diffusions with measure-dependent and possibly
degenerate diffusion matrix studied in Budhiraja, Dupuis and Fischer
\cite{budhirajaetal12}. The variational form of the Laplace principle
rate function established there is shown to be expressible in relative
entropy form. As a by-product, one obtains a variational representation
of relative entropy with respect to Wiener measure. The \hyperref[app]{Appendix}
contains two results regarding relative entropy: the contraction
property mentioned above, which extends a well-known invariance
property (Appendix \ref{AppREContraction}), and a direct proof of the
variational representation of relative entropy with respect to Wiener
measure (Appendix \ref{AppREWiener}). In Appendix~\ref{AppItoSuff},
easily verifiable conditions entailing the hypotheses of the Laplace
principle of Section \ref{SectIto} are given.

%s2 #&#
\section{Basic definitions and results} \label{SectBasics}

Let $\mathcal{S}$ be a Polish space (i.e., a separable topological
space metrizable with a complete metric). Denote by $\mathcal
{B}(\mathcal{S})$ the $\sigma$-algebra of Borel subsets of $\mathcal
{S}$ and by
$\mathcal{P}(\mathcal{S})$ the space of probability measures on
$\mathcal{B}(\mathcal{S})$ equipped with the topology of weak
convergence. For
$\mu,\nu\in\mathcal{P}(\mathcal{S})$, let $R(\nu\| \mu)$ denote the
\emph{relative entropy} of $\nu$ with respect to $\mu$, that is,
\[
R(\nu\|\mu) \doteq\cases{\displaystyle \int_{\mathcal{S}} \log\biggl(
\frac{\mrmdd\nu}{\mrmdd\mu}(x) \biggr)\nu(\mrmdd x), &\quad if $\nu$ absolutely continuous
w.r.t. $
\mu$,
\vspace*{2pt}\cr
\infty, &\quad else.}
\]
Relative entropy is well defined as a $[0,\infty]$-valued function, it
is lower semicontinuous as a function of both variables, and $R(\nu\|
\mu) = 0$ if and only if $\nu=\mu$.

Let $(\xi^{n})_{n\in\mathbb{N}}$ be a sequence of $\mathcal
{S}$-valued random variables. A \emph{rate function} on $\mathcal{S}$
is a lower semicontinuous function $\mathcal{S} \rightarrow[0,\infty
]$. Let $I$ be a rate function on $\mathcal{S}$. By lower
semicontinuity, the sublevel sets of $I$, that is, the sets
$I^{-1}([0,c])$ for $c\in[0,\infty)$, are closed. A rate function is
said to be \emph{good} if its sublevel sets are compact.

%de2.1 #&#
\begin{defn}
The sequence $(\xi^{n})_{n\in\mathbb{N}}$ satisfies the \emph{large
deviation principle} with rate function $I$ if for all $B\in\mathcal
{B}(\mathcal{S})$,
\begin{eqnarray*}
-\inf_{x\in B^{\circ}} I(x) &\leq& \liminf_{n\to\infty}
\frac
{1}{n} \log\Prb\bigl\{\xi^{n} \in B \bigr\}
\\
&\leq& \limsup_{n\to\infty} \frac{1}{n} \log\Prb\bigl\{
\xi^{n} \in B \bigr\} \leq- \inf_{x\in\cl(B)} I(x),
\end{eqnarray*}
where $\cl(B)$ denotes the closure and $B^{\circ}$ the interior of $B$.
\end{defn}

%de2.2 #&#
\begin{defn}
The sequence $(\xi^{n})$ satisfies the \emph{Laplace principle} with
rate function $I$ if for all $G\in\mathbf{C}_{b}(\mathcal{S})$,
\[
\lim_{n\to\infty} -\frac{1}{n} \log\Mean\bigl[\exp\bigl(-n
\cdot G\bigl(\xi^{n}\bigr) \bigr) \bigr] = \inf_{x\in\mathcal{S}}
\bigl\{I(x) + G(x) \bigr\},
\]
where $\mathbf{C}_{b}(\mathcal{S})$ denotes the space of all
bounded continuous functions $\mathcal{S}\rightarrow\mathbb{R}$.
\end{defn}

Clearly, the large deviation principle (or Laplace principle) is a
distributional property. The rate function of a large deviation
principle is unique; see, for instance, Lemma 4.1.4 in
Dembo and Zeitouni \cite{dembozeitouni98}, page 117. The large
deviation principle holds
with a good rate function if and only if the Laplace principle holds
with a good rate function, and the rate function is the same; see, for
instance, Theorem 4.4.13 in Dembo and Zeitouni \cite{dembozeitouni98},
page 146.

The fact that, for good rate functions, the large deviation principle
implies the Laplace principle is a consequence of Varadhan's integral
lemma; see Theorem 3.4 in Varadhan \cite{varadhan66}. Another
consequence of
Varadhan's lemma is the first of the following two basic transfer
results, given here as Theorem \ref{ThTiltedLDP}; cf. Theorem II.7.2
in Ellis \cite{ellis85}, page 52.

%th2.1 #&#
\begin{thrm}[(Change of measure, Varadhan)] \label{ThTiltedLDP}
Let $(\xi^{n})$ be a sequence of $\mathcal{S}$-valued random
variables such that $(\xi^{n})$ satisfies the large deviation
principle with good rate function $I$. Let $(\tilde{\xi}^{n})_{n\in
\mathbb{N}}$ be a second sequence of $\mathcal{S}$-valued random
variables. Suppose that, for every $n\in\mathbb{N}$, $\Law(\tilde
{\xi}^{n})$ is absolutely continuous with respect to $\Law(\xi^{n})$
with density
\[
\frac{\mrmdd\Law(\tilde{\xi}^{n})}{\mrmdd\Law(\xi^{n})}(x) = \exp\bigl(n\cdot
F(x) \bigr),\qquad x\in\mathcal{S},
\]
where $F\dvtx\mathcal{S}\rightarrow\mathbb{R}$ is continuous and such that
\[
\lim_{L\to\infty} \limsup_{n\to\infty} \frac{1}{n}
\log\Mean\bigl[\mathbf{1}_{[L,\infty)}\bigl(F\bigl(\xi^{n}\bigr)
\bigr)\cdot\exp\bigl(n\cdot F\bigl(\xi^{n}\bigr) \bigr) \bigr] = -
\infty.
\]
Then $(\tilde{\xi}^{n})_{n\in\mathbb{N}}$ satisfies the large
deviation principle with good rate function $I - F$.
\end{thrm}

The second basic transfer result is the contraction principle, given
here as Theorem \ref{ThContraction}; see, for instance, Theorem 4.2.1
and Remark (c) in Dembo and Zeitouni \cite{dembozeitouni98},
pages 126 and 127.

%th2.2 #&#
\begin{thrm}[(Contraction principle)] \label{ThContraction}
Let $(\xi^{n})$ be a sequence of $\mathcal{S}$-valued random
variables such that $(\xi^{n})$ satisfies the large deviation
principle with good rate function $I$. Let $\psi\dvtx\mathcal{S}
\rightarrow\mathcal{Y}$ be a measurable function, $\mathcal{Y}$ a
Polish space. If $\psi$ is continuous on $I^{-1}([0,\infty))$, then
$(\psi(\xi^{n}))$ satisfies the large deviation principle with good
rate function
\[
J(y) \doteq\inf_{x\in\psi^{-1}(y)} I(x),\qquad y\in\mathcal{Y},
\]
where $\inf\varnothing= \infty$ by convention.
\end{thrm}

Let $X_{1},X_{2},\ldots$ be $\mathcal{S}$-valued independent and
identically distributed random variables with common distribution $\mu
\in\mathcal{P}(\mathcal{S})$ defined on some probability space
$(\Omega,\mathcal{F},\Prb)$. For $n\in\mathbb{N}$, let $\mu^{n}$ be the
\emph{empirical measure} of $X_{1},\ldots,X_{n}$, that is,
\[
\mu^{n}(\omega) \doteq\frac{1}{n} \sum
_{i=1}^{n} \delta_{X_{i}(\omega)}, \qquad\omega\in\Omega,
\]
where $\delta_{x}$ denotes the Dirac measure concentrated in $x \in
\mathcal{S}$. Sanov's theorem gives the large deviation principle for
$(\mu^{n})_{n\in\mathbb{N}}$ in terms of relative entropy. For a
proof, see, for instance, Section~6.2 in
Dembo and Zeitouni \cite{dembozeitouni98}, pages 260--266, or
Chapter 2 in Dupuis and Ellis \cite{dupuisellis97}, pages 39--52.
Recall that $\mathcal{P}(\mathcal{S})$ is equipped with the topology of
weak convergence of measures.

%th2.3 #&#
\begin{thrm}[(Sanov)] \label{ThSanov}
The sequence $(\mu^{n})_{n\in\mathbb{N}}$ of $\mathcal{P}(\mathcal
{S})$-valued random variables satisfies the large deviation principle
with good rate function
\[
I(\theta) \doteq R ( \theta\| \mu),\qquad \theta\in\mathcal{P}(\mathcal{S}).
\]
\end{thrm}

We are interested in analogous results for the empirical measures of
weakly interacting systems. For $N\in\mathbb{N}$, let
$X^{N}_{1},\ldots,X^{N}_{N}$ be $\mathcal{S}$-valued random variables
defined on some probability space
$(\Omega_{N},\mathcal{F}_{N},\Prb_{N})$. Denote by $\mu^{N}$ the
empirical measure of $X^{N}_{1},\ldots,X^{N}_{N}$.
%
%de2.3 #&#
\begin{defn} \label{DefMeanField}
The triangular array $(X^{N}_{i})_{N\in\mathbb{N},i\in\{1,\ldots,N\}
}$ is called a \emph{weakly interacting system} if the following hold:
\begin{longlist}
\item for each $N\in\mathbb{N}$, $X^{N}_{1},\ldots,X^{N}_{N}$ is a
finite exchangeable sequence;
\item the family $(\mu^{N})_{N\in\mathbb{N}}$ of $\mathcal
{P}(\mathcal{S})$-valued random variables is tight.
\end{longlist}
\end{defn}

Recall that a finite sequence $Y_{1},\ldots,Y_{N}$ of random variables
with values in a common measurable space is called \emph{exchangeable}
if its joint distribution is invariant under permutations of the
components, that is, $\Law(Y_{1},\ldots,Y_{N}) = \Law(Y_{\sigma
(1)},\ldots,Y_{\sigma(N)})$ for every permutation $\sigma$ of $\{
1,\ldots,N\}$. A weakly interacting system $(X^{N}_{i})$ is said to
satisfy the \emph{law of large numbers} if there exists $\mu\in
\mathcal{P}(\mathcal{S})$ such that $(\mu^{N})$ converges to $\mu$ in
distribution or, equivalently, $(\Prb_{N}\circ(\mu^{N})^{-1})$
converges weakly to $\delta_{\mu}$. Weakly interacting systems are
sometimes called \emph{mean field systems}. In the situation of
Theorem \ref{ThSanov}, setting $X^{N}_{i}\doteq X_{i}$, $N\in\mathbb
{N}$, $i\in\{1,\ldots,N\}$, defines a weakly interacting system that
satisfies the law of large numbers, the limit measure being the common
sample distribution.

%s3 #&#
\section{A toy model and the desired form of the rate function}
\label{SectRightForm}

For $N\in\mathbb{N}$, let $(Y^{N}_{i}(t))_{i\in\{1,\ldots,N\},t\in
\{0,1\}}$ be an independent family of standard normal real random
variables on some probability space $(\Omega,\mathcal{F},\Prb)$. Let
$b\dvtx\mathbb{R} \rightarrow\mathbb{R}$ be measurable; below we will
assume $b$ to be bounded and continuous. Define real random variables
$X^{N}_{1}(t),\ldots,X^{N}_{N}(t)$, $t\in\{0,1\}$, by
%
%e3.1 #&#
\begin{equation}
\label{ExToySystem} X^{N}_{i}(0) \doteq
Y^{N}_{i}(0),\qquad X^{N}_{i}(1) \doteq
X^{N}_{i}(0) + \frac{1}{N}\sum
_{j=1}^{N} b \bigl(X^{N}_{j}(0)
\bigr) + Y^{N}_{i}(1).
\end{equation}
We may interpret the variables $X^{N}_{i}(t)$ as the states of the
components of an $N$-particle system at times $t\in\{0,1\}$. This toy
model can be obtained as the first two steps in a discrete time version
of a system of weakly interacting It{\^o} diffusions; cf. the
discussion following Example \ref{ExmplDiscreteTimeModel} below. Let
$\mu^{N}$ be the empirical measure of the $N$-particle system on
``path space,'' that is,
\[
\mu^{N}_{\omega} \doteq\frac{1}{N} \sum
_{i=1}^{N} \delta_{X^{N}_{i}(\omega)} = \frac{1}{N}
\sum_{i=1}^{N}
\delta_{(X^{N}_{i}(0,\omega),X^{N}_{i}(1,\omega))},\qquad
\omega
\in\Omega.
\]
Notice that the components of $X^{N}$ are identically distributed and
interact only through $\mu^{N}$ since
\[
\frac{1}{N}\sum_{j=1}^{N} b
\bigl(X^{N}_{j}(0) \bigr) = \int_{\mathbb{R}^{2}}
b(x) \mrmd \mu^{N}(x,\tilde{x})
\]
and the variables $Y^{N}_{i}(t)$ are independent and identically
distributed. The sequence $X^{N}_{1},\ldots,X^{N}_{N}$ of $\mathbb
{R}^{2}$-valued random variables is exchangeable.

Let $\lambda^{N}$ denote the empirical measure of $
 Y^{N}_{1},\ldots,Y^{N}_{N} $. By Sanov's theorem, $(\lambda
^{N})_{N\in\mathbb{N}}$ satisfies the large deviation principle with
good rate function $R(\cdot\|\gamma_{0})$, where $\gamma_{0}$ is the
bivariate standard normal distribution. Following the usual way of
deriving the large deviation principle, we observe that, for every
$N\in\mathbb{N}$, the law of $\mu^{N}$ is absolutely continuous with
respect to the law of $\lambda^{N}$. To see this, set, for
$\mathbf{y},\tilde{\mathbf{y}}\in\mathbb{R}^{N}$, $\theta
\in\mathcal{P}(\mathbb{R}^{2})$,
\begin{eqnarray*}
\nu^{N}_{(\mathbf{y},\tilde{\mathbf{y}})} &\doteq& \frac
{1}{N}\sum
_{i=1}^{N} \delta_{(y_{i},\tilde{y}_{i})},\qquad m_{b}(
\theta)\doteq\int b(x) \mrmd \theta(x,\tilde{x}),
\\[-1.5pt]
\nu^{N}_{\mathbf{y}} &\doteq& \frac{1}{N}\sum
_{i=1}^{N} \delta_{y_{i}},\qquad
\mathbf{m}_{b}(\mathbf{y})\doteq{ \biggl(\int b(x)
\nu^{N}_{\mathbf{y}}(\mrmdd x),\ldots,\int b(x) \nu^{N}_{\mathbf{y}}(\mrmdd x)
\biggr)}^\mathsf{T}.
\end{eqnarray*}
Define functions $f\dvtx\mathcal{P}(\mathbb{R}^{2})\times\mathbb{R}
\rightarrow\mathbb{R}$ and $F\dvtx\mathcal{P}(\mathbb
{R}^{2})\rightarrow
[-\infty,\infty)$ according to
\begin{eqnarray*}
f\bigl(\theta,(y,\tilde{y})\bigr) &\doteq& \bigl(y + m_{b}(\theta)
\bigr)\cdot\tilde{y} - \frac{1}{2} \bigl\llvert y + m_{b}(\theta)
\bigr\rrvert^{2},
\\[-1.5pt]
F(\theta) &\doteq& \cases{\displaystyle \int_{\mathbb{R}^{2}} f\bigl(\theta,(y,\tilde
{y})\bigr) \mrmd \theta(y,\tilde{y}), &\quad if $f(\theta,\cdot)$ is $\theta$-integrable,
\vspace*{1pt}\cr
-\infty, &\quad otherwise.}
\end{eqnarray*}
Then the law of $X^{N}$ is absolutely continuous with respect to the
law of $Y^{N}$ with density given by
%
%e3.2 #&#
\begin{eqnarray}
\label{EqToyProcessDensity} \frac{\mrmdd\Law(X^{N})}{\mrmdd\Law
(Y^{N})}(\mathbf{y},\tilde{\mathbf{y}}) &=& \exp\biggl( \bigl\langle\mathbf{y} + \mathbf{m}_{b}(
\mathbf{y}), \tilde{\mathbf{y}}\bigr\rangle-\frac{1}{2} \bigl|
\mathbf{y} + \mathbf{m}_{b}(\mathbf{y})\bigr|^{2}
\biggr)
\\
&=& \exp\bigl(N\cdot F \bigl(\mu^{N}_{(\mathbf{y},\tilde{\mathbf{y}})} \bigr) \bigr).
\nonumber
\end{eqnarray}
Since $\mu^{N} = \nu^{N}_{(X^{N}(0),X^{N}(1))}$ and $\lambda^{N} =
\nu^{N}_{(Y^{N}(0),Y^{N}(1))}$, it follows from
(\ref{EqToyProcessDensity}) that
%
%e3.3 #&#
\begin{equation}
\label{EqToyDensity} \frac{\mrmdd\Law(\mu^{N})}{\mrmdd\Law(\lambda^{N})}(\theta)
= \exp\bigl( N\cdot F(\theta)
\bigr),\qquad \theta\in\mathcal{P}\bigl(\mathbb{R}^{2}\bigr).
\end{equation}

The densities given by (\ref{EqToyDensity}) are of the form
required by Theorem \ref{ThTiltedLDP}, the change of measure version
of Varadhan's lemma. Assume from now on that $b$ is bounded and
continuous. Then $F$ is upper semicontinuous and the tail condition in
Theorem \ref{ThTiltedLDP} is satisfied. However, $F$ is discontinuous
at any $\theta\in\mathcal{P}(\mathbb{R}^{2})$ such that $F(\theta) >
-\infty$. Indeed, let $\eta$ be the univariate standard Cauchy
distribution and set $\theta_{n}\doteq(1-\frac{1}{n})\theta+ \frac
{1}{n}\delta_{0}\otimes\eta$, $n\in\mathbb{N}$. Then $\theta_{n}
\to\theta$ weakly, while $F(\theta_{n}) = -\infty$ for all $n$.

Although Theorem \ref{ThTiltedLDP} cannot be applied directly, an
approximation argument based on Varadhan's lemma could be used to show
(cf. Remark \ref{RemToyTiltedLDP} below) that the sequence of
empirical measures $(\mu^{N})_{N\in\mathbb{N}}$ satisfies the large
deviation principle with good rate function
%
%e3.4 #&#
\begin{equation}
\label{EqToyRateFnctAux} I(\theta) \doteq R(\theta\|\gamma_{0}) - F(
\theta),\qquad \theta\in\mathcal{P}\bigl(\mathbb{R}^{2}\bigr).
\end{equation}
The function $I$ in (\ref{EqToyRateFnctAux}) can be rewritten in terms
of relative entropy as follows. Define a mapping
$\mathbb{R}\dvtx\mathcal
{P}(R^{2})\times\mathbb{R}^{2} \rightarrow\mathbb{R}^{2}$ by
%
%e3.5 #&#
\begin{equation}
\label{ExToyModelpsi} \psi\bigl(\theta,(y,\tilde{y})\bigr) \doteq\bigl
(y, y +
m_{b}(\theta) + \tilde{y}\bigr).
\end{equation}
For $\theta\in\mathcal{P}(\mathbb{R}^{2})$, let $\Psi_{\gamma
_{0}}(\theta)$ be the image measure of $\gamma_{0}$ under $\psi
(\theta,\cdot)$. Then $\Psi_{\gamma_{0}}(\theta)$ is equivalent
to $\gamma_{0}$ with density given by
\[
\frac{\mrmdd \Psi_{\gamma_{0}}(\theta)}{\mrmdd \gamma_{0}}(y,\tilde{y}) = \exp\bigl
(f\bigl(\theta,(y,\tilde{y})\bigr)
\bigr).%= \exp\left( \left(y +
%m_{b}(\theta)\right)\cdot\tilde{y} - \frac{1}{2}|y+m_{b}(\theta)|^{2}
\]
If $\theta$ is not absolutely continuous with respect to $\Psi
_{\gamma_{0}}(\theta)$, then $R(\theta\| \Psi_{\gamma_{0}}(\theta
)) = \infty= R(\theta\| \gamma_{0})$. If $\theta$ is absolutely
continuous with respect to $\Psi_{\gamma_{0}}(\theta)$, then
\begin{eqnarray*}
R \bigl(\theta\| \Psi_{\gamma_{0}}(\theta) \bigr) &=& \int\log\biggl(
\frac{\mrmdd \theta}{\mrmdd\Psi_{\gamma_{0}}(\theta)} \biggr) \mrmd \theta
\\
&=& \int\log\biggl(\frac{\mrmdd\theta}{\gamma_{0}} \biggr) \mrmd \theta- \int\log
\biggl(
\frac{\mrmdd \Psi_{\gamma_{0}}(\gamma_{0})}{\mrmdd \gamma
_{0}} \biggr) \mrmd \theta
\\
&=& R (\theta\| \gamma_{0} ) - F(\theta).
\end{eqnarray*}
Consequently, for all $\theta\in\mathcal{P}(\mathbb{R}^{2})$,
%
%e3.6 #&#
\begin{equation}
\label{EqToyRateFnct} I(\theta) = R \bigl(\theta\| \Psi_{\gamma
_{0}}(\theta)
\bigr).
\end{equation}
Notice that $\Psi_{\gamma_{0}}(\theta)$ is the law of a one-particle
system with measure variable frozen at $\theta$; $\Psi_{\gamma
_{0}}(\theta)$ can also be interpreted as the solution of the
McKean--Vlasov equation for the toy model with measure variable frozen
at $\theta$.

%re3.1 #&#
\begin{rem} \label{RemToyTiltedLDP}
A version of Varadhan's lemma (or Theorem \ref{ThTiltedLDP}) that
allows to rigorously derive the large deviation principle for $(\mu
^{N})$ with rate function in relative entropy form is provided by
Lemma 1.1 in Del Moral and Zajic \cite{delmoralzajic03}. Observe that
the density of $\Law(X^{N})$ may be computed with respect to product
measures different from $\Law(Y^{N})=\otimes^{N}\gamma_{0}$. A
natural alternative is the product $\otimes^{N}\Psi_{\gamma_{0}}(\mu
_{\ast})$, where $\mu_{\ast}$ is the (unique) solution of the fixed
point equation $\mu= \Psi_{\gamma_{0}}(\mu)$; $\mu_{\ast}$ can be
seen as the McKean--Vlasov distribution of the toy model. We do not give
the details here. The results of Section \ref{SectNoiseBased}, based
on different arguments, will imply that $(\mu^{N})_{N\in\mathbb{N}}$
satisfies the large deviation principle with good rate function $I$ as
given by (\ref{EqToyRateFnct}); see Example \ref
{ExmplToyModel} below.
\end{rem}

%re3.2 #&#
\begin{rem} \label{RemFormRateFnct}
Equation (\ref{EqToyRateFnct}) gives the desired form of the rate
function in terms of relative entropy. More generally, suppose that
$\Psi\dvtx\mathcal{P}(\mathcal{S}) \rightarrow\mathcal{P}(\mathcal
{S})$ is
continuous, where $\mathcal{S}$ is a Polish space. Then the function
\[
J(\theta)\doteq R \bigl(\theta\| \Psi(\theta) \bigr),\qquad \theta\in\mathcal{P}(
\mathcal{S}),
\]
is lower semicontinuous with values in $[0,\infty]$, hence a rate
function, and it is in relative entropy form. The lower semicontinuity
of $J$ follows from the lower semicontinuity of relative entropy
jointly in both its arguments and the continuity of $\Psi$. If, in
addition, $\range(\Psi) \doteq\{\Psi(\theta)\dvt\theta\in\mathcal
{P}(\mathcal{S})\}$ is compact in $\mathcal{P}(\mathcal{S})$, then the
sublevel sets of $J$ are compact and $J$ is a good rate function.
Indeed, compactness of $\range(\Psi)$ implies tightness, and the
compactness of the sublevel sets of $J$, which are closed by lower
semicontinuity, follows as in the proof of Lemma 1.4.3(c) in
Dupuis and Ellis \cite{dupuisellis97}, pages 29--31.
\end{rem}

%s4 #&#
\section{Noise-based systems} \label{SectNoiseBased}

Let $\mathcal{X}$, $\mathcal{Y}$ be Polish spaces. For $N\in\mathbb
{N}$, let $X^{N}_{1},\ldots,X^{N}_{N}$ be $\mathcal{X}$-valued random
variables defined on some probability space $(\Omega_{N},\mathcal
{F}_{N},\Prb_{N})$. Denote by $\mu^{N}$ the empirical measure of
$X^{N}_{1},\ldots,X^{N}_{N}$. We suppose that there are a probability
measure $\gamma_{0}\in\mathcal{P}(\mathcal{Y})$ and a Borel measurable
mapping $\psi\dvtx\mathcal{P}(\mathcal{X})\times\mathcal{Y} \rightarrow
\mathcal{X}$ such that the following representation for the triangular
array $(X^{N}_{i})_{i\in\{1,\ldots,N\},N\in\mathbb{N}}$ holds: for each
$N\in\mathbb{N}$, there is a sequence $Y^{N}_{1},\ldots,Y^{N}_{N}$ of
independent and identically distributed $\mathcal {Y}$-valued random
variables on $(\Omega_{N},\mathcal{F}_{N},\Prb _{N})$ with common
distribution $\gamma_{0}$ such that for all $i\in\{ 1,\ldots,N\}$,
%
%e4.1 #&#
\begin{equation}
\label{EqNoiseRep} X^{N}_{i}(\omega) = \psi\bigl(
\mu^{N}_{\omega},Y^{N}_{i}(\omega) \bigr),\qquad
\Prb_{N}\mbox{-almost all }\omega\in\Omega_{N}.
\end{equation}
The above representation entails by symmetry that, for $N$ fixed, the
sequence $X_{1}^{N},\ldots,X^{N}_{N}$ is exchangeable. Representation
(\ref{EqNoiseRep}) also implies that $\mu^{N}$ satisfies the equation
%
%e4.2 #&#
\begin{equation}
\label{EqPrelimitFP} \mu^{N} = \frac{1}{N} \sum
_{i=1}^{N} \delta_{\psi(\mu
^{N},Y^{N}_{i})},\qquad \Prb_{N}
\mbox{-almost surely.}
\end{equation}

In order to describe the limit behavior of the sequence of empirical
measures $(\mu^{N})_{N\in\mathbb{N}}$, define a mapping $\Psi\dvtx
\mathcal{P}(\mathcal{Y})\times\mathcal{P}(\mathcal{X}) \rightarrow
\mathcal{P}(\mathcal{X})$ by
%
%e4.3 #&#
\begin{equation}
\label{ExLimitModel} (\gamma,\mu) \mapsto\Psi_{\gamma}(\mu) \doteq\gamma
\circ\psi^{-1}(\mu,\cdot).
\end{equation}
Thus $\Psi_{\gamma}(\mu)$ is the image measure of $\gamma$ under
the mapping $\mathcal{Y} \ni y \mapsto\psi(\mu,y)$. Equivalently,
$\Psi_{\gamma}(\mu) = \Law(\psi(\mu,Y))$ with $Y$ any $\mathcal
{Y}$-valued random variable with distribution $\gamma$. Limit points
of $(\mu^{N})_{N\in\mathbb{N}}$ will be described in terms of
solutions to the fixed point equation
%
%e4.4 #&#
\begin{equation}
\label{EqLimitFP} \mu= \Psi_{\gamma}(\mu).
\end{equation}

Assume that there is a Borel measurable set $\mathcal{D} \subset
\mathcal{P}(\mathcal{Y})$ such that the following properties hold:
\begin{longlist}
\item[(A1)] %\label{HypUnique}
Equation (\ref{EqLimitFP}) has a unique
fixed point $\mu_{\ast}(\gamma)$ for every $\gamma\in\mathcal
{D}$, and
the mapping $\mathcal{D}\ni\gamma\mapsto\mu_{\ast}(\gamma) \in
\mathcal{P}(\mathcal{X})$ is Borel measurable.

\item[(A2)] %\label{HypPrelimit}
For all $N\in\mathbb{N}$,
\[
\otimes^{N}\gamma_{0} \Biggl\{(y_{1},\ldots,y_{N})\in\mathcal{Y}^{N}\dvt\frac{1}{N}\sum
_{i=1}^{N}\delta_{y_{i}} \in
\mathcal{D} \Biggr\} = 1.
\]
\item[(A3)] %\label{HypLimit}
If $\gamma\in\mathcal{P}(\mathcal{Y})$ is
such that $R(\gamma\|\gamma_{0}) < \infty$, then $\gamma\in
\mathcal{D}$ and $\mu_{\ast|\mathcal{D}}$ is continuous at $\gamma
$.
\end{longlist}

Assumption (A2) implies that (\ref{EqLimitFP}) possesses a
unique solution for almost all (with respect to products of
$\gamma_{0}$) probability measures of empirical measure form. Such
probability measures are therefore in the domain of definition of the
mapping $\mu_{\ast|\mathcal{D}}$. According to
assumption (A3), also all probability measures $\gamma$
with finite $\gamma_{0}$-relative entropy are in the domain of
definition of $\mu_{\ast|\mathcal{D}}$, which is continuous at any such
$\gamma$ in the topology of weak convergence.
%
%th4.1 #&#
\begin{thrm} \label{ThNoiseLDP}
Grant \textup{(A1)--(A3)}. Then the sequence $(\mu^{N})_{N\in\mathbb
{N}}$ satisfies the large deviation principle with good rate function
$I\dvtx\mathcal{P}(\mathcal{X}) \rightarrow[0,\infty]$ given by
\[
I(\eta) = \inf_{\gamma\in\mathcal{D}:\mu_{\ast}(\gamma)=\eta} R (\gamma
\| \gamma_{0} ),
\]
where $\inf\varnothing= \infty$ by convention.
\end{thrm}

\begin{pf} The assertion follows from Sanov's theorem and the
contraction principle. To see this, let $\lambda^{N}$ denote the
empirical measure of $Y^{N}_{1},\ldots,Y^{N}_{N}$. Then for $\Prb
_{N}$-almost all $\omega\in\Omega_{N}$,
%
%e4.5 #&#
\begin{equation}
\label{EqEmpiricalFP} \mu^{N}_{\omega} = \frac{1}{N} \sum
_{i=1}^{N} \delta_{\psi(\mu
^{N}_{\omega},Y^{N}_{i}(\omega))} =
\lambda^{N}_{\omega}\circ\psi^{-1} \bigl(
\mu^{N}_{\omega},\cdot\bigr) = \Psi_{\lambda
^{N}_{\omega}} \bigl(
\mu^{N}_{\omega} \bigr).
\end{equation}
Thus, $\mu^{N} = \Psi_{\lambda^{N}}(\mu^{N})$ with probability one. For
$\Prb_{N}$-almost all $\omega\in\Omega_{N}$, $\lambda^{N}_{\omega}
\in\mathcal{D}$ by assumption (A2) and, by uniqueness according to
(A1), $\mu_{\ast}(\lambda^{N}_{\omega}) =
\mu^{N}_{\omega}$. By Theorem \ref{ThSanov} (Sanov),
$(\lambda^{N})_{N\in\mathbb{N}}$ satisfies the large deviation
principle with good rate function $R(\cdot\|\gamma_{0})$. By
assumption (A3), $\mu_{\ast}(\cdot)$~is defined and continuous
on $\{\gamma\in\mathcal{P}(\mathcal{Y})\dvt R(\gamma\| \gamma_{0}) <
\infty
\}$. Theorem \ref{ThContraction} (contraction principle) therefore
applies, and it follows that
$(\mu_{\ast}(\lambda^{N}))_{N\in\mathbb{N}}$, hence
$(\mu^{N})_{N\in\mathbb{N}}$, satisfies the large deviation principle
with good rate function
\[
\mathcal{P}(\mathcal{X})\ni\eta\mapsto\inf_{\gamma\in\mathcal{D}:
\mu_{\ast}(\gamma)=\eta} R (\gamma\|
\gamma_{0} ).
\]
\upqed\end{pf}

The rate function of Theorem \ref{ThNoiseLDP} can be expressed in
relative entropy form as in Remark \ref{RemFormRateFnct}. The key
observation is the contraction property of relative entropy established
in Lemma \ref{LemmaRE} in the \hyperref[app]{Appendix}.

%co4.2 #&#
\begin{crll} \label{CorNoiseRateFnct}
Let $I$ be the rate function of Theorem \ref{ThNoiseLDP}. Then for all
$\eta\in\mathcal{P}(\mathcal{X})$,
\[
I(\eta) = R \bigl(\eta\| \Psi_{\gamma_{0}}(\eta) \bigr).
\]
\end{crll}

\begin{pf} Let $\eta\in\mathcal{P}(\mathcal{X})$. The mapping
$\mathcal{P}(\mathcal{Y}) \ni\gamma\mapsto\Psi_{\gamma}(\eta)
\in\mathcal{P}(\mathcal{X})$ is Borel measurable. Since $\{\gamma
\in\mathcal{P}(\mathcal{X})\dvt R(\gamma\|\gamma_{0}) < \infty\}
\subset\mathcal
{D}$ and $\inf\varnothing= \infty$,
\[
\inf_{\gamma\in\mathcal{D}: \mu_{\ast}(\gamma)=\eta} R (\gamma\| \gamma
_{0} ) = \inf
_{\gamma\in\mathcal{D}:
\Psi_{\gamma}(\eta)=\eta} R (\gamma\| \gamma_{0} ) = \inf
_{\gamma\in\mathcal{P}(\mathcal{Y}): \Psi_{\gamma}(\eta
)=\eta} R (\gamma\| \gamma_{0} ).
\]
By Lemma \ref{LemmaRE}, it follows that
\[
\inf_{\gamma\in\mathcal{P}(\mathcal{Y}): \Psi_{\gamma}(\eta
)=\eta} R (\gamma\| \gamma_{0} ) = R \bigl(\eta
\| \Psi_{\gamma
_{0}}(\eta) \bigr).
\]
\upqed\end{pf}

%ex4.1 #&#
\begin{exmpl} \label{ExmplToyModel}
Consider the toy model of Section \ref{SectRightForm}. Suppose that
$b\in\mathbf{C}_{b}(\mathbb{R})$. Then $\theta\mapsto
m_{b}(\theta) \doteq\int b(x) \mrmd \theta(x,\tilde{x})$ is bounded and
continuous as a mapping $\mathcal{P}(\mathbb{R}^{2})\rightarrow
\mathbb
{R}$. Observe\vspace*{1pt} that $m_{b}(\theta)$ depends only on the first marginal
of $\theta$. Set $\mathcal{X}\doteq\mathbb{R}^{2}$, $\mathcal
{Y}\doteq\mathbb{R}^{2}$, let $\gamma_{0}$ be the bivariate standard
normal distribution, and define $\psi\dvtx\mathcal{P}(\mathbb
{R}^{2})\times\mathbb{R}^{2} \rightarrow\mathbb{R}^{2}$ according
to (\ref{ExToyModelpsi}). Recalling (\ref{ExToySystem}), one
sees that the toy model satisfies representation (\ref{EqNoiseRep}).
Based on $\psi$, define $\Psi$ according to (\ref{ExLimitModel}).
Given any $\gamma\in\mathcal{P}(\mathbb{R}^{2})$, the mapping $\mu
\mapsto\Psi_{\gamma}(\mu)$ possesses a unique fixed point $\mu
_{\ast}(\gamma)$. To see this, suppose that $\theta\in\mathcal
{P}(\mathbb{R}^{2})$ is a fixed point, that is, $\theta=\Psi_{\gamma
}(\theta) = \gamma\circ\psi(\theta,\cdot)^{-1}$. Let
$X=(X(0),X(1))$, $Y = (Y(0),Y(1))$ be two $\mathbb{R}^{2}$-valued
random variables on some probability space $(\Omega,\mathcal{F},\Prb
)$ with distribution $\theta$ and $\gamma$, respectively. By the
fixed point property, $\Law(X) = \Law(\psi(\theta,Y))$. By
definition of $\psi$, $\Law(X(0)) = \Law(Y(0))$. Since $m_{b}(\theta
)$ depends on $\theta= \Law(X)$ only through its first marginal,
which is equal to $\Law(X(0)) = \Law(Y(0))$, we have $m_{b}(\theta)
= m_{b}(\gamma)$. It follows that, for all $B_{0}, B_{1}\in\mathcal
{B}(\mathbb{R})$,
\[
\Prb\bigl(X(1) \in B_{1} | X(0) \in B_{0} \bigr) = \Prb
\bigl(Y(0) + m_{b}(\gamma) + Y(1) \in B_{1} | Y(0) \in
B_{0} \bigr).
\]
This determines the conditional distribution of $X(1)$ given $X(0)$
and, since $\Law(X(0)) = \Law(Y(0))$, also the joint law of $X(0)$ and
$X(1)$. In fact, $\Law(X) = \Psi_{\gamma}(\gamma)$. Consequently,
$\mu_{\ast}(\gamma) \doteq\Psi_{\gamma}(\gamma)$ is the unique
solution of (\ref{EqLimitFP}). By the extended mapping theorem
for weak convergence (Theorem 5.5 in Billingsley \cite{billingsley68},
page 34) and
since $m_{b}(\cdot)\in\mathbf{C}_{b}(\mathcal{P}(\mathbb{R}^{2}))$, the
mapping $(\gamma,\mu) \mapsto\Psi_{\gamma}(\gamma)$ is continuous
as a
function $\mathcal{P}(\mathbb{R}^{2})\times\mathcal{P}(\mathbb{R}^{2})
\rightarrow\mathcal{P}(\mathbb{R}^{2})$. It follows that the mapping
$\gamma\mapsto\mu_{\ast}(\gamma) = \Psi_{\gamma}(\gamma)$ is
continuous. Assumptions \textup{(A1)--(A3)} are therefore satisfied
with the
choice $\mathcal{D}\doteq\mathcal{P}(\mathbb{R}^{2})$. By
Corollary \ref{CorNoiseRateFnct}, the sequence of empirical measures
$(\mu^{N})$ for the toy model satisfies the large deviation principle
with good rate function $I$ given by (\ref{EqToyRateFnct}). Observe
that the distribution $\gamma_{0}$ need not be the bivariate standard
normal distribution for the large deviation principle to hold; it can
be any probability measure on $\mathcal{B}(\mathbb{R}^{2})$.
\end{exmpl}

%ex4.2 #&#
\begin{exmpl} \label{ExmplToyModelAlt}
Consider the following variation on the toy model of Section \ref
{SectRightForm} and Example \ref{ExmplToyModel}. For $N\in\mathbb
{N}$, let $(Y^{N}_{i}(t))_{i\in\{1,\ldots,N\},t\in\{0,1\}}$ be
independent standard normal real random variables as above. Denote by
$\gamma_{0}$ the bivariate standard normal distribution and let $B\in
\mathcal{B}(\mathbb{R})$ be a $\gamma_{0}$-continuity set, that is,
$\gamma_{0}(\partial( B\times\mathbb{R})) = 0$, where $\partial(
B\times\mathbb{R})$ is the boundary of $B\times\mathbb{R}$. Define
real random variables $X^{N}_{1}(t),\ldots,X^{N}_{N}(t)$, $t\in\{0,1\}
$, by
\[
X^{N}_{i}(0) \doteq Y^{N}_{i}(0),\qquad
X^{N}_{i}(1) \doteq X^{N}_{i}(0) +
\frac{1}{N} \Biggl(\sum_{j=1}^{N}
\mathbf{1}_{B}\bigl(X^{N}_{j}(0)\bigr)
\Biggr) \cdot Y^{N}_{i}(1).
\]
For this new toy model, define $\psi\dvtx\mathcal{P}(\mathbb
{R}^{2})\times
\mathbb{R}^{2} \rightarrow\mathbb{R}^{2}$ by
\[
\psi\bigl(\mu,(y,\tilde{y})\bigr)\doteq\bigl(y, y + \mu(B\times\mathbb
{R})\cdot
\tilde{y}\bigr).
\]
With this choice of $\psi$, representation (\ref{EqNoiseRep}) holds and
$\psi$ is measurable as composition of measurable maps since
$\mu\mapsto\mu(B\times\mathbb{R})$ is measurable with respect to the
Borel $\sigma$-algebra induced by the topology of weak convergence.
Based on $\psi$, define $\Psi$ according to (\ref{ExLimitModel}). As in
Example \ref{ExmplToyModel}, one checks that the fixed point equation
(\ref{EqLimitFP}) possesses a unique solution $\mu_{\ast}(\gamma)
\doteq\Psi_{\gamma}(\gamma)$ for every $\gamma\in
\mathcal{P}(\mathbb{R}^{2})$. However, if $\partial(B\times\mathbb
{R})\neq
\varnothing$, then $\mu_{\ast}(\cdot)$ is not continuous on
$\mathcal{P}(\mathbb{R}^{2})$. On the other hand, if $\gamma\in
\mathcal{P}(\mathbb{R}^{2})$ is such that $R(\gamma\| \gamma_{0}) <
\infty$, then $\gamma$ is absolutely continuous with respect to
$\gamma_{0}$, so that $B\times\mathbb{R}$ is also a $\gamma$-continuity
set. By the extended mapping theorem, it follows that $\mu_{\ast}(\cdot)$
is continuous at any such $\gamma$. Assumptions \textup{(A1)--(A3)} are
therefore satisfied, again with the choice $\mathcal{D}\doteq
\mathcal{P}(\mathbb{R}^{2})$, and Corollary \ref{CorNoiseRateFnct} yields
the large deviation principle. In this example, if
$\gamma_{0}(B\times\mathbb{R}) < 1$, then the distribution of
$\mu^{N}$, the empirical measure of $X^{N}_{1},\ldots,X^{N}_{N}$, is
not absolutely continuous with respect to $\lambda^{N}$, the empirical
measure of $Y^{N}_{1},\ldots,Y^{N}_{N}$. Indeed, in this case, the
event $\{ \nu^{N}_{(\mathbf{y},\mathbf{y})}\dvtx\mathbf{y}\in
\mathbb{R}^{N}\}\subset\mathcal{P}(\mathbb{R}^{2})$, where
$\nu^{N}_{(\mathbf{y},\mathbf{y})}$ is defined as in
Section \ref{SectRightForm}, has strictly positive probability with
respect to $\Prb\circ(\mu^{N})^{-1}$, while it has probability zero
with respect to $\Prb\circ(\lambda^{N})^{-1}$.
\end{exmpl}

%ex4.3 #&#
\begin{exmpl}[(Discrete time systems)] \label{ExmplDiscreteTimeModel}
Let $T\in\mathbb{N}$. Let $\mathcal{X}_{0}$, $\mathcal{Y}_{0}$ be
Polish spaces, and let $\mathcal{X}$, $\mathcal{Y}$ be the Polish
product spaces $\mathcal{X}\doteq(\mathcal{X}_{0})^{T+1}$ and
$\mathcal{Y}\doteq(\mathcal{Y}_{0})^{T+1}$, respectively. Let
\[
\phi_{0}\dvtx\mathcal{Y}_{0} \rightarrow
\mathcal{X}_{0},\qquad \phi\dvtx\{1,\ldots,T\}\times\mathcal{X}_{0}
\times\mathcal{P}(\mathcal{X}_{0})\times\mathcal{Y}_{0}
\rightarrow\mathcal{X}_{0}
\]
be measurable maps. Let $\gamma_{0}\in\mathcal{P}(\mathcal{Y})$
and, for
$N\in\mathbb{N}$, let $Y^{N}_{1},\ldots,Y^{N}_{N}$ be independent
and identically distributed $\mathcal{Y}$-valued random variables
defined on some probability space $(\Omega_{N},\mathcal{F}_{N},\Prb
_{N})$ with common distribution $\gamma_{0}$. Write $Y^{N}_{i} =
(Y^{N}_{i}(t))_{t\in\{0,\ldots,T\}}$ and define $\mathcal{X}$-valued
random variables $X^{N}_{1},\ldots,X^{N}_{N}$ with $X^{N}_{i} =
(X^{N}_{i}(t))_{t\in\{0,\ldots,T\}}$ recursively by
%
%e4.6 #&#
\begin{eqnarray}
\label{ExDiscreteTimeModel} X^{N}_{i}(0)&\doteq&
\phi_{0} \bigl(Y^{N}_{i}(0) \bigr),
\\
X^{N}_{i}(t+1)&\doteq&\phi\bigl(t+1,X^{N}_{i}(t),
\mu^{N}(t),Y^{N}_{i}(t+1) \bigr),\qquad t\in\{0,\ldots,T
- 1\},
\nonumber
\end{eqnarray}
where $\mu^{N}(t)\doteq\frac{1}{N}\sum_{i=1}^{N}\delta _{X^{N}_{i}(t)}$
is the empirical measure of $X^{N}_{1},\ldots,X^{N}_{N}$ at marginal
(or time)~$t$. In analogy with (\ref {ExDiscreteTimeModel}), define
$\psi\dvtx\mathcal{P}(\mathcal {X})\times \mathcal{Y}
\rightarrow\mathcal{X}$ according to $(\mu,y)=
(\mu,(y_{0},\ldots,y_{T}))\mapsto\psi(\mu,y)\doteq x$ with $x =
(x_{0},\ldots,x_{T})$ given by
%
%e4.7 #&#
\begin{eqnarray}
\label{ExDiscreteTimeModelpsi} x_{0}&\doteq&\phi_{0}(y_{0}),
\nonumber\\[-8pt]\\[-8pt]
x_{t+1}&\doteq&\phi\bigl(t+1,x_{t},\mu(t),y_{t+1}
\bigr),\qquad t\in\{ 0,\ldots,T - 1\},
\nonumber
\end{eqnarray}
where $\mu(t)$ is the marginal of $\mu\in\mathcal{P}(\mathcal{X})$ at
time $t$. Then $\psi$ is measurable as a composition of measurable
maps, and representation (\ref{EqNoiseRep}) holds. Based on $\psi$,
define $\Psi$ according to (\ref{ExLimitModel}). Using the recursive
structure of (\ref{ExDiscreteTimeModel}) and the components of $\psi$
according to (\ref{ExDiscreteTimeModelpsi}), one checks that the fixed
point equation (\ref{EqLimitFP}) has a unique solution $\mu_{\ast
}(\gamma)$ given any $\gamma\in\mathcal{D}\doteq\mathcal
{P}(\mathcal{Y})$. To be more precise, define functions $\bar{\phi
}_{t}\dvtx\mathcal{P}(\mathcal{X}_{0})^{t}\times\mathcal{Y}
\rightarrow\mathcal
{X}_{0}$, $t\in\{0,\ldots,T\}$, recursively by
%
%e4.8 #&#
\begin{eqnarray}
\label{ExDiscreteTimeModelphibar} \bar{\phi}_{0}(y)&\doteq&
\phi_{0}(y_{0}),
\nonumber\\[-8pt]\\[-8pt]
\bar{\phi}_{t} \bigl((\alpha_{0},\ldots,
\alpha_{t-1}),y \bigr)&\doteq&\phi\bigl(t,\bar{\phi}_{t-1}
\bigl((\alpha_{0},\ldots,\alpha_{t-2}),y\bigr),
\alpha_{t-1}, y_{t} \bigr).
\nonumber
\end{eqnarray}
Notice that $\bar{\phi}_{t}$ depends on $y = (y_{0},\ldots,y_{T})\in
\mathcal{Y}$ only through $(y_{0},\ldots,y_{t})$. Given $\gamma\in
\mathcal{P}(\mathcal{Y})$, recursively define probability measures
$\alpha
_{t}(\gamma) \in\mathcal{P}(\mathcal{X}_{0})$, $t\in\{0,\ldots,T\}$,
according to
%
%e4.9 #&#
\begin{eqnarray}
\label{ExDiscreteTimeModelalpha} \alpha_{0}(\gamma) &\doteq& \gamma
\circ\bar{\phi}_{0}^{-1},
\nonumber\\[-8pt]\\[-8pt]
\alpha_{t}(\gamma) &\doteq& \gamma\circ\bar{\phi}_{t} \bigl(
\bigl(\alpha_{0}(\gamma),\ldots,\alpha_{t-1}(\gamma)\bigr),
\cdot\bigr)^{-1},\qquad t\in\{1,\ldots,T\}.
\nonumber
\end{eqnarray}
The mapping $\mathcal{P}(\mathcal{Y})\ni\gamma\mapsto\alpha
_{t}(\gamma
)\in\mathcal{P}(\mathcal{X}_{0})$ is measurable for every $t\in\{
0,\ldots,T\}$. Define $\Phi\dvtx\mathcal{P}(\mathcal
{X}_{0})^{T}\times
\mathcal{Y} \rightarrow\mathcal{X}$ by
%
%e4.10 #&#
\begin{equation}
\label{ExDiscreteTimeModelPhi} \Phi\bigl((\alpha_{0},\ldots,
\alpha_{T-1}),y \bigr)\doteq\bigl(\bar{\phi}_{0}(y), \bar{
\phi}_{1}(\alpha_{0},y),\ldots, \bar{\phi}_{T}
\bigl((\alpha_{0},\ldots,\alpha_{T-1}),y\bigr) \bigr).
\end{equation}
Then the mapping $\gamma\mapsto\gamma\circ\Phi((\alpha_{0}(\gamma
),\ldots,\alpha_{T-1}(\gamma)),\cdot)^{-1}$ is measurable and
provides the unique fixed point of (\ref{EqLimitFP}) with
noise distribution $\gamma\in\mathcal{P}(\mathcal{Y})$. In fact,
%
%e4.11 #&#
\begin{equation}
\label{EqDiscreteTimeModelFP} \mu_{\ast}(\gamma) = \gamma\circ\Phi
\bigl(\bigl(\alpha_{0}(\gamma),\ldots,\alpha_{T-1}(\gamma)
\bigr),\cdot\bigr)^{-1}.
\end{equation}
Writing $\mu_{\ast}(t,\gamma)$ for the $t$-marginal of $\mu_{\ast
}(\gamma)$, we also notice that
\[
\mu_{\ast}(t,\gamma) = \alpha_{t}(\gamma) = \gamma\circ\bar{
\phi}_{t} \bigl(\bigl(\alpha_{0}(\gamma),\ldots,
\alpha_{t-1}(\gamma)\bigr),\cdot\bigr)^{-1}.
\]
If $\phi_{0}$, $\phi$ are continuous maps, then it follows by (\ref
{EqDiscreteTimeModelFP}) and the extended mapping theorem that $\mu
_{\ast}(\cdot)$ is continuous on $\mathcal{D} = \mathcal{P}(\mathcal{Y})$,
and Corollary \ref{CorNoiseRateFnct} yields the large deviation
principle for the sequence of ``path space'' empirical measures $(\mu
^{N})_{N\in\mathbb{N}}$.
\end{exmpl}

Example \ref{ExmplDiscreteTimeModel} comprises a large class of
discrete time weakly interacting systems. The sequence of $(\mathcal
{X}_{0})^{N}$-valued random variables $X^{N}(0),\ldots,X^{N}(T)$ given
by (\ref{ExDiscreteTimeModel}) enjoys the Markov property if the
$(\mathcal{Y}_{0})^{N}$-valued random variables
$Y^{N}(0),\ldots,Y^{N}(T)$ are independent (since
$Y^{N}_{1},\ldots,Y^{N}_{N}$ are assumed to be independent and
identically distributed with common distribution $\gamma_{0}$, this
amounts to requiring that $\gamma _{0}$ be of product form, that is,
$\gamma_{0} = \bigotimes^{T}_{t=0}\nu _{t}$ for some
$\nu_{0},\ldots,\nu_{T}\in\mathcal{P}(\mathcal {Y}_{0})$). In
particular, discrete time versions of weakly interacting It{\^o}
processes as considered in Section \ref{SectIto} are covered by Example
\ref{ExmplDiscreteTimeModel}. More precisely, assuming coefficients of
diffusion type and using a standard Euler--Maruyana scheme for the
system of stochastic differential equations (\ref {EqSDEPrelimit}) and
the corresponding limit equation (\ref {EqSDELimit}), one would choose
$T\in\mathbb{N}$ and $h > 0$ so that $h\cdot T$ corresponds to the
continuous time horizon, set $\mathcal {X}_{0}\doteq\mathbb{R}^{d}$,
$\mathcal{Y}_{0}\doteq\mathbb {R}^{d_{1}}$, define
$\phi\dvtx\{1,\ldots,T\}\times\mathbb
{R}^{d}\times\mathcal{P}(\mathcal{X}_{0})\times\mathcal{Y}_{0}
\rightarrow\mathcal{X}_{0}$ according to
\[
\phi(t,x,\nu,y)\doteq x + \tilde{b} \bigl((t - 1)h, x, \nu\bigr)h +
\sqrt{h}
\cdot\tilde{\sigma} \bigl((t - 1)h, x, \nu\bigr)y,
\]
and set $\gamma_{0}\doteq\otimes^{T+1}\nu$ for some $\nu\in
\mathcal{P}(\mathbb{R}^{d_{1}})$ with mean zero and identity
covariance matrix
(in particular, $\nu\doteq N(0,\Id_{d_{1}})$ the $d_{1}$-variate
standard normal distribution). In Section \ref{SectIto} we assume for
simplicity that all component processes have the same deterministic
initial condition; this corresponds to setting $\phi_{0}\equiv x_{0}$
for some $x_{0}\in\mathbb{R}^{d}$. If the drift coefficient $b$ and
the dispersion coefficient $\sigma$ are continuous, then so is $\phi
$, and Corollary \ref{CorNoiseRateFnct} applies.

Example \ref{ExmplDiscreteTimeModel} also applies to finite state
discrete time weakly interacting Markov chains, which arise as discrete
time versions of the mean field systems found, for instance, in the
analysis of large communication networks, especially \mbox{WLANs}
(cf. Duffy \cite{duffy10}, for an overview). In this situation, the functions
$\phi_{0}$, $\phi$ are in general discontinuous in $y\in
\mathcal{Y}_{0}$; yet the hypotheses of
Corollary \ref{CorNoiseRateFnct} are still satisfied. To be more
precise, let $\mathcal{S}\doteq\{s_{1},\ldots,s_{M}\}$ be a finite
set, and let $\iota\dvtx\mathcal{S} \rightarrow\{1,\ldots,M\}$ be the
natural bijection between elements of $\mathcal{S}$ and their indices
(thus $\iota(s_{i}) = i$ for every $i\in\{1,\ldots,M\}$). The space of
probability measures $\mathcal{P}(\mathcal{S})$ can be identified with
$\{p\in[0,1]^M \dvt\sum_{k=1}^{M} p_{k} = 1\}$ endowed with the standard
metric. For $t\in\mathbb{N}_{0}$, let $a_{ij}(t,\cdot)\dvtx
\mathcal{P}(\mathcal{S}) \rightarrow[0,1]$, $i,j\in\{1,\ldots,M\}$,
be measurable maps such that, for every $p\in
\mathcal{P}(\mathcal{S})$, $A(t,p) \doteq(a_{ij}(t,p))_{i,j\in
\{1,\ldots,M\}}$ is a transition probability matrix on $\mathcal{S}
\equiv\{1,\ldots,M\}$. Let $q\in\mathcal{P}(\mathcal{S})$. Using the
notation of Example \ref{ExmplDiscreteTimeModel}, fix $T\in
\mathbb{N}$, set $\mathcal{X}_{0}\doteq\mathcal{S}$,
$\mathcal{Y}_{0}\doteq[0,1]$, $\mathcal{X}\doteq
\mathcal{X}_{0}^{T+1}$, and $\mathcal{Y}\doteq\mathcal{Y}_{0}^{T+1}$;
define $\phi\dvtx\{1,\ldots,T\}\times\mathcal{X}_{0}\times
\mathcal{P}(\mathcal{X}_{0}) \times\mathcal{Y}_{0} \rightarrow
\mathcal{X}_{0}$ by
\[
\phi(t,x,p,y)\doteq\sum_{j=1}^{M}
s_{j}\cdot\mathbf{1}_{(\sum
_{k=1}^{j-1}a_{\iota(x)k}(t-1,p), \sum_{k=1}^{j}a_{\iota(x)k}(t-1,p)]}(y),
\]
and let $\phi_{0}\dvtx\mathcal{Y}_{0} \rightarrow\mathcal{X}_{0}$ be
given by
\[
\phi_{0}(y)\doteq\sum_{j=1}^{M}
s_{j}\cdot\mathbf{1}_{(\sum
_{k=1}^{j-1}q_{k}, \sum_{k=1}^{j}q_{k}]}(y).
\]
Set $\gamma_{0}\doteq\otimes^{T+1} \lambda_{[0,1]}$ with $\lambda
_{[0,1]}$ Lebesgue measure on $\mathcal{B}([0,1]) = \mathcal
{B}(\mathcal{Y}_{0})$. For $N\in\mathbb{N}$, let
$Y^{N}_{1},\ldots,Y^{N}_{N}$ be independent and identically distributed
$\mathcal{Y}$-valued random variables with common distribution
$\gamma_{0}$, and define $\mathcal {X}$-valued random variables
$X^{N}_{1},\ldots,X^{N}_{N}$ with $X^{N}_{i} =
(X^{N}_{i}(t))_{t\in\{0,\ldots,T\}}$ recursively by
(\ref{ExDiscreteTimeModel}). Observe that
$X^{N}_{1}(0),\ldots,X^{N}_{N}(0)$ are independent and identically
distributed with common distribution $q$. Moreover, for all
\mbox{$t\in\{0,\ldots,T-1\}$}, all $\mathbf{z}\in\mathcal{S}^{N}$,
%
%e4.12 #&#
\begin{eqnarray}
\label{EqAppMCComponent}
&&\Prb_{N} \bigl( X^{N}(t+1)=
\mathbf{z} | X^{N}(0),\ldots,X^{N}(t) \bigr)\nonumber \\[-0.2pt]
&&\quad= \prod
_{i=1}^{N} a_{\iota(X^{N}_{i}(t))\iota
(z_{i})}\bigl(t,
\mu^{N}(t)\bigr)
\\[-0.2pt]
&&\quad= \exp\biggl(N\cdot\int_{\mathcal{S}} \log\bigl( a_{\iota
(x)\iota(z_{i})}
\bigl(t,\mu^{N}(t)\bigr) \bigr) \mu^{N}(t,\mrmdd x) \biggr),
\nonumber
\end{eqnarray}
where\vspace*{1pt} $\log(0) = -\infty$, $\RMe^{-\infty} = 0$. It follows that
$(X^{N}(t))_{t\in\{0,\ldots,T\}}$ is a Markov chain with state space
$\mathcal{S}^{N}$. Equation (\ref{EqAppMCComponent}) also implies
that $(\mu^{N}(t))_{t\in\{0,\ldots,T\}}$ is a Markov chain with
transition probabilities given by
\begin{eqnarray*}
&&\Prb_{N} \Biggl( \mu^{N}(t+1)= \frac{1}{N}\sum
_{i=1}^{N}\delta_{z_{i}} \Big|
X^{N}(0),\ldots,X^{N}(t) \Biggr)
\\[-0.2pt]
&&\quad= \sum_{\tilde{\mathbf{z}}\in\mathfrak{p}(\mathbf{z})} \exp
\biggl( N\cdot\int
_{\mathcal{S}} \log\bigl(a_{\iota(x)\iota
(\tilde{z}_{i})}\bigl(t,\mu^{N}(t)
\bigr) \bigr) \mu^{N}(t,\mrmdd x) \biggr),
\end{eqnarray*}
where $\mathfrak{p}(\mathbf{z})$ indicates the set of elements of
$\mathcal{S}^{N}$ that arise by permuting the components of
$\mathbf{z}\in\mathcal{S}^{N}$. For $i\in\{1,\ldots,N\}$, again by
(\ref{EqAppMCComponent}), the process couple
$((X^{N}_{i}(t),\mu^{N}(t)))_{t\in\{0,\ldots,T\}}$ is a Markov chain
with state space $\mathcal{S}\times\mathcal{P}(\mathcal{S})$, and its
law does not depend on the component $i$. Define the function $\psi$
according to (\ref{ExDiscreteTimeModelpsi}), and define $\Psi$
according to (\ref{ExLimitModel}). As in the more general situation of
Example \ref{ExmplDiscreteTimeModel}, equation (\ref{EqLimitFP}) (i.e.,
the fixed point equation $\Psi_{\gamma}(\mu) = \mu$) has a unique
solution $\mu_{\ast}(\gamma)$ given any $\gamma\in\mathcal
{D}\doteq\mathcal{P}(\mathcal{Y})$, representation (\ref
{EqDiscreteTimeModelFP}) holds for $\mu_{\ast}(\gamma)$, and the
mapping $\mathcal{P}(\mathcal{Y}) \ni\gamma\mapsto\mu_{\ast }(\gamma)
\in\mathcal{P}(\mathcal{X})$ is measurable. Let us assume that the maps
$p\mapsto a_{ij}(t,p)$ are continuous for all $i,j\in\{1,\ldots,M\}$,
$t\in\mathbb{N}_{0}$. In order to verify the hypotheses of Corollary
\ref{CorNoiseRateFnct}, it then remains to check that $\mu _{\ast}(\cdot)$
is continuous at any $\tilde{\gamma}\in\mathcal {P}(\mathcal{Y})$ such
that $R(\tilde{\gamma}|\gamma_{0}) < \infty $. To do this, take
$\tilde{\gamma}\in\mathcal{P}(\mathcal{Y})$ absolutely continuous with
respect to $\gamma_{0}$, and let $(\tilde{\gamma }_{n})
\subset\mathcal{P}(\mathcal{Y})$ be such that $\tilde{\gamma }_{n}
\to\tilde{\gamma}$ as $n\to\infty$. Recall (\ref
{ExDiscreteTimeModelphibar}), the definition of the functions $\bar
{\phi}_{t}$, and (\ref{ExDiscreteTimeModelalpha}), the definition of
the maps $\gamma\mapsto\alpha_{t}(\gamma)$. For $t\in\{0,\ldots,T\}$
set
\begin{eqnarray*}
D_{t}&\doteq&\bigl\{y \in\mathcal{Y}\dvt\exists\bigl(y^{n}
\bigr)_{n\in
\mathbb{N}}\subset\mathcal{Y} \mbox{ such that, as $n\to\infty$, }
y^{n}\to y \mbox{ but }
\\[-0.2pt]
&&\hspace*{5pt}\bar{\phi}_{t} \bigl(\bigl(\alpha_{0}(\tilde{
\gamma}_{n}),\ldots,\alpha_{t-1}(\tilde{\gamma}_{n})
\bigr),y^{n} \bigr) \nrightarrow\bar{\phi}_{t} \bigl(\bigl(
\alpha_{0}(\tilde{\gamma}),\ldots,\alpha_{t-1}(\tilde{
\gamma})\bigr),y \bigr) \bigr\}.
\end{eqnarray*}
By definition of $\bar{\phi}_{0}$ and $\phi_{0}$, we have
\[
D_{0} \subseteq\Biggl\{y \in\mathcal{Y}\dvt y_{0}\in
\Biggl\{ \sum_{k=1}^{j} q_{k}
\dvt j\in\{0,\ldots,M\} \Biggr\} \Biggr\}.
\]
It follows that $\gamma_{0}(D_{0}) = 0$ and, since $\tilde{\gamma}$
is absolutely continuous with respect to $\gamma_{0}$, $\tilde{\gamma
}(D_{0}) = 0$. The extended mapping theorem implies that $\alpha
_{0}(\tilde{\gamma}_{n}) \to\alpha_{0}(\tilde{\gamma})$ as $n\to
\infty$. Using this convergence, the definition of $\bar{\phi}_{1}$
in terms of $\phi$, the continuity of $p\mapsto a_{ij}(t,p)$, and the
fact that $\bar{\phi}_{0}$ is continuous on $\mathcal{Y}\setminus
D_{0}$, we find that
\[
D_{1} \subseteq D_{0}\cup\Biggl\{y \in\mathcal{Y}\dvt
y_{1}\in\Biggl\{ \sum_{k=1}^{j}
a_{\iota(\bar{\phi}_{0}(y))k}\bigl(0,\alpha_{0}(\tilde{\gamma})\bigr
)\dvt j\in
\{0,\ldots,M\} \Biggr\} \Biggr\}.
\]
Since $\bar{\phi}_{0}(y)$ depends on $y$ only through $y_{0}$ (in
fact, $\bar{\phi}_{0}(y) = \phi_{0}(y_{0})$), it follows that
$\gamma_{0}(D_{1}) = 0$, hence $\tilde{\gamma}(D_{1}) = 0$. The
extended mapping theorem in the version of Theorem 5.5 in
Billingsley~\cite{billingsley68}, page 34, implies that $\alpha
_{1}(\tilde{\gamma}_{n})\to\alpha_{1}(\tilde{\gamma})$ as $n\to
\infty$. Proceeding by induction over $t$, one checks that
\begin{eqnarray*}
D_{t} &\subseteq& D_{0}\cup\cdots\cup D_{t-1}\\
&&{}\cup
\Biggl\{y \in\mathcal{Y}\dvt
y_{t}\in\Biggl\{ \sum_{k=1}^{j}
a_{\iota(\bar{\phi}_{t-1}((\alpha
_{0}(\tilde{\gamma}),\ldots,\alpha_{t-1}(\tilde{\gamma
})),y))k}\bigl(t-1,\alpha_{t-1}(\tilde{\gamma})\bigr)\dvt j\in
\{0,\ldots,M\} \Biggr\} \Biggr\}
\end{eqnarray*}
and, since\vspace*{1pt}
$\bar{\phi}_{t-1}((\alpha_{0}(\tilde{\gamma}),\ldots,\alpha_{t-1}(\tilde{\gamma})),y)$
depends on $y$ only through the components $(y_{0},\ldots, y_{t-1})$,
$\gamma_{0}(D_{t}) = 0 = \tilde {\gamma}(D_{t})$, which implies that
$\alpha_{t}(\tilde{\gamma }_{n})\to\alpha_{t}(\tilde{\gamma})$ as
$n\to\infty$. Set $D\doteq\bigcup_{t=0}^{T} D_{t}$ and recall (\ref
{ExDiscreteTimeModelPhi}), the definition of $\Phi$. Let $y \in
\mathcal{Y}$, $(y^{n})_{n\in\mathbb{N}}\subset\mathcal{Y}$ be such that
$y^{n}\to y$ as $n\to\infty$. Then
\[
\Phi\bigl(\bigl(\alpha_{0}(\tilde{\gamma}_{n}),\ldots,
\alpha_{T-1}(\tilde{\gamma}_{n})\bigr),y^{n} \bigr)
\stackrel{n\to\infty} {\longrightarrow} \Phi\bigl(\bigl(\alpha_{0}(
\tilde{\gamma}),\ldots,\alpha_{T-1}(\tilde{\gamma})\bigr),y \bigr)
\qquad\mbox{if } y\notin D.
\]
Since $\gamma_{0}(D) = 0 = \tilde{\gamma}(D)$, the extended mapping
theorem yields
\[
\tilde{\gamma}_{n}\circ\Phi\bigl(\bigl(\alpha_{0}(\tilde{
\gamma}_{n}),\ldots,\alpha_{T-1}(\tilde{\gamma}_{n})
\bigr),\cdot\bigr)^{-1} \stackrel{n\to\infty} {\longrightarrow} \tilde{
\gamma}\circ\Phi\bigl(\bigl(\alpha_{0}(\tilde{\gamma}),\ldots,
\alpha_{T-1}(\tilde{\gamma})\bigr),\cdot\bigr)^{-1}.
\]
Recalling representation (\ref{EqDiscreteTimeModelFP}) we conclude that
\[
\mu_{\ast}(\tilde{\gamma}_{n}) \stackrel{n\to\infty} {
\longrightarrow} \mu_{\ast}(\tilde{\gamma}),
\]
which establishes continuity of $\mu_{\ast}(\cdot)$ at any $\tilde
{\gamma}$ with $R(\tilde{\gamma}\|\gamma_{0}) < \infty$ since any
such $\tilde{\gamma}$ is absolutely continuous with respect to
$\gamma_{0}$. Under the assumption that the maps $p\mapsto
a_{ij}(t,p)$ are continuous, we have thus derived the large deviation
principle for $(\mu^{N})_{N\in\mathbb{N}}$ with rate function $\eta
\mapsto R(\eta\| \Psi_{\gamma_{0}}(\eta))$; here $\Psi_{\gamma
_{0}}(\eta)$ coincides with the law of a time-inhomogeneous $\mathcal
{S}$-valued Markov chain with initial distribution $q$ and transition
matrices $A(t,\eta(t))$, $t\in\{0,\ldots,T-1\}$. The same arguments
and a completely analogous construction work for weakly interacting
Markov chains with countably infinite state space $\mathcal{S}$.
Notice that we need not require the transition probabilities
$a_{ij}(t,p)$ to be bounded away from zero; in particular, whether
$a_{ij}(t,p)$ is equal to zero or strictly positive may depend on the
measure variable $p$.

%s5 #&#
\section{Weakly interacting It{\^o} processes} \label{SectIto}

In this section, we consider weakly interacting systems described by
It{\^o} processes as studied in Budhiraja, Dupuis and Fischer \cite
{budhirajaetal12}. We show that the Laplace principle rate function
derived there in variational form can be expressed in non-variational
form in terms of relative entropy. We do not give the most general
conditions under which the results hold; in particular, we assume here
that all particles obey the same deterministic initial condition.

Let $T > 0$ be a finite time horizon, let $d,d_{1}\in\mathbb{N}$, and
let $x_{0}\in\mathbb{R}^{d}$. Set $\mathcal{X}\doteq\mathbf
{C}([0,T],\mathbb{R}^{d})$, $\mathcal{Y}\doteq\mathbf
{C}([0,T],\mathbb{R}^{d_{1}})$, equipped with the maximum norm
topology. Let $b$, $\sigma$ be predictable functionals defined on
$[0,T]\times\mathcal{X}\times\mathcal{P}(\mathbb{R}^{d})$ with values
in $\mathbb{R}^{d}$ and $\mathbb{R}^{d\times d_{1}}$, respectively. For
$N\in\mathbb{N}$, let $((\Omega_{N},\mathcal{F}^{N},\Prb
_{N}),(\mathcal{F}^{N}_{t}))$ be a stochastic basis satisfying the
usual hypotheses and carrying $N$ independent $d_{1}$-dimensional
$(\mathcal{F}^{N}_{t}))$-Wiener processes $W^{N}_{1},\ldots,W^{N}_{N}$.
The $N$-particle system is described by the solution to the system of
stochastic differential equations
%
%e5.1 #&#
\begin{equation}
\label{EqSDEPrelimit} \mrmdd X^{N}_{i}(t) = b
\bigl(t,X^{N}_{i},\mu^{N}(t) \bigr)\mrmd t + \sigma
\bigl(t,X^{N}_{i},\mu^{N}(t) \bigr)\mrmd W^{N}_{i}(t)
\end{equation}
with initial condition $X^{N}_{i}(0) = x_{0}$, $i\in\{1,\ldots,N\}$,
where $\mu^{N}(t)$ is the empirical measure of
$X^{N}_{1},\ldots,X^{N}_{N}$ at time $t\in[0,T]$, that is,
\[
\mu^{N}(t,\omega)\doteq\frac{1}{N}\sum
_{i=1}^{N}\delta_{X^{N}_{i}(t,\omega)},\qquad \omega\in
\Omega_{N}.
\]
The coefficients $b$, $\sigma$ in (\ref{EqSDEPrelimit}) may
depend on the entire history of the solution trajectory, not only its
current value as in the diffusion case. In the diffusion case, in fact,
one has $b(t,\phi,\nu) = \tilde{b}(t,\phi(t),\nu)$, $\sigma
(t,\phi,\nu) = \tilde{\sigma}(t,\phi(t),\nu)$ for some functions
$\tilde{b}$, $\tilde{\sigma}$ defined on $[0,T]\times\mathbb
{R}^{d}\times\mathcal{P}(\mathbb{R}^{d})$, and the solution process
$X^{N}$ is a Markov process with state space $\mathbb{R}^{N\times d}$.

Denote by $\mu^{N}$ the empirical measure of
$(X^{N}_{1},\ldots,X^{N}_{N})$ over the time interval $[0,T]$, that is,
$\mu^{N}$ is the $\mathcal{P}(\mathcal{X})$-valued random variable
defined by
\[
\mu_{\omega}^{N}\doteq\frac{1}{N}\sum
_{i=1}^{N}\delta_{X^{N}_{i}(\cdot,\omega)},\qquad \omega\in
\Omega_{N}.
\]
%
%Clearly, the distribution of $\mu^{N}(t)$ is identical to the marginal
%distribution of $\mu^{N}$ at time $t$, i.e., $\mu^{N}(t) = \mu^{N}\circ
%time $t$.
The asymptotic behavior of $\mu^{N}$ as $N$ tends to infinity can be
characterized in terms of solutions to the ``non-linear'' stochastic
differential equation
%
%e5.2 #&#
\begin{equation}
\label{EqSDELimit} \mrmdd X(t) = b \bigl(t,X,\Law\bigl(X(t)\bigr) \bigr)\mrmd t +
\sigma
\bigl(t,X,\Law\bigl(X(t)\bigr) \bigr)\mrmd W(t)
\end{equation}
with $\Law(X(0)) = \delta_{x_{0}}$, where $W$ is a standard
$d_{1}$-dimensional Wiener process defined on some stochastic basis.
Notice that the law of the solution itself appears in the coefficients
of (\ref{EqSDELimit}). In the diffusion case, the
corresponding Kolmogorov forward equation is therefore a non-linear
parabolic partial differential equation, and it corresponds to the
McKean--Vlasov equation of the weakly interacting system defined by
(\ref{EqSDEPrelimit}).

For the statement of the Laplace principle, we need to consider
controlled versions of (\ref{EqSDEPrelimit}) and (\ref
{EqSDELimit}), respectively. For $N\in\mathbb{N}$, let $\mathcal
{U}_{N}$ be the space of all $(\mathcal{F}^{N}_{t})$-progressively
measurable functions $u\dvtx[0,T]\times\Omega_{N} \rightarrow\mathbb
{R}^{N\times d_{1}}$ such that
\[
\Mean_{N} \Biggl[ \sum_{i=1}^{N}
\int_{0}^{T}\bigl|u_{i}(t)\bigr|^{2}\mrmd t
\Biggr] <\infty,
\]
where $u=(u_{1},\ldots,u_{N})$ and $\Mean_{N}$ denotes expectation
with respect to $\Prb_{N}$. Given $u\in\mathcal{U}_{N}$, the
counterpart of (\ref{EqSDEPrelimit}) is the system of
controlled stochastic differential equations
%
%e5.3 #&#
\begin{eqnarray}
\label{EqSDEPrelimitControl} \mrmdd \bar{X}^{N}_{i}(t) &=& b
\bigl(t,\bar{X}^{N}_{i},\bar{\mu}^{N}(t) \bigr)\mrmd t
+ \sigma\bigl(t,\bar{X}^{N}_{i},\bar{\mu
}^{N}(t) \bigr)u_{i}(t)\mrmd t
\nonumber\\[-8pt]\\[-8pt]
&&{}+\sigma\bigl(t,\bar{X}^{N}_{i},\bar{\mu}^{N}(t)
\bigr)\mrmd W^{N}_{i}(t)
\nonumber
\end{eqnarray}
with initial condition $\bar{X}^{N}_{i}(0)=x_{0}$, where $\bar{\mu
}^{N}(t)$ denotes the empirical measure of
$\bar{X}^{N}_{1},\ldots,\bar{X}^{N}_{N}$ at time $t$.

Let $\mathcal{U}$ be the set of quadruples $((\Omega,\mathcal
{F},\Prb),(\mathcal{F}_{t}),u,W)$ such that the pair
$((\Omega,\mathcal{F},\Prb),(\mathcal{F}_{t}))$ forms a stochastic
basis satisfying the usual hypotheses, $W$ is a $d_{1}$-dimensional
$(\mathcal{F}_{t})$-Wiener process, and $u$ is an $\mathbb
{R}^{d_{1}}$-valued $(\mathcal{F}_{t})$-progressively measurable
process such that
\[
\Mean\biggl[\int_{0}^{T} \bigl|u(t)\bigr|^{2}\mrmd t
\biggr] < \infty.
\]
For simplicity, we may write $u\in\mathcal{U}$ instead of
$((\Omega,\mathcal{F},\Prb),(\mathcal{F}_{t}),u,W)\in\mathcal{U}$.
Given $u\in\mathcal{U}$, the counterpart of (\ref{EqSDELimit})
is the controlled ``non-linear'' stochastic differential equation
%
%e5.4 #&#
\begin{eqnarray}
\label{EqSDELimitControl} \mrmdd \bar{X}(t) &=& b \bigl(t,\bar{X},\Law\bigl(\bar{X}(t)
\bigr) \bigr)\mrmd t + \sigma\bigl(t,\bar{X},\Law\bigl(\bar{X}(t)\bigr)
\bigr)u(t)\mrmd t
\nonumber\\[-8pt]\\[-8pt]
&&{}+\sigma\bigl(t,\bar{X},\Law\bigl(\bar{X}(t)\bigr) \bigr)\mrmd W(t)
\nonumber
\end{eqnarray}
with initial condition $\Law(\bar{X}(0)) = \delta_{x_{0}}$. A solution
of (\ref{EqSDELimitControl}) under $u\in\mathcal {U}$ is a continuous
$\mathbb{R}^{d}$-valued process $\bar{X}$ defined on the given
stochastic basis and adapted to the given filtration such that the
integral version of (\ref{EqSDELimitControl}) holds with probability
one. Denote by $\mathcal {R}_{1}$ the space of deterministic relaxed
controls with finite first moments, that is, $\mathcal{R}_{1}$ is the
set of all positive measures on
$\mathcal{B}(\mathbb{R}^{d_{1}}\times[0,T])$ such that
$r(\mathbb{R}^{d_{1}}\times[0,t]) = t$ for all $t\in[0,T]$ and $\int
_{\mathbb{R}^{d_{1}}\times[0,T]}|y| r(\mrmdd y \times \mrmdd t)<\infty$. Equip
$\mathcal{R}_{1}$ with the topology of weak convergence of measures
plus convergence of first moments. Let $u\in\mathcal{U}$. The joint
distribution of $(u,W)$ can be identified with a probability measure on
$\mathcal{B}(\mathcal{R}_{1}\times\mathcal{Y})$. If $\bar {X}$ is a
solution of (\ref{EqSDELimitControl}) under $u$, then the
joint distribution of $(\bar{X},u,W)$ can be identified with a
probability measure on $\mathcal{B}(\mathcal{Z})$, where $\mathcal
{Z}\doteq\mathcal{X}\times\mathcal{R}_{1}\times\mathcal{Y}$.

%de5.1 #&#
\begin{defn} \label{DefWeakUnique}
\emph{Weak uniqueness} of solutions is said to hold for
(\ref{EqSDELimitControl}) if, whenever $u, \tilde{u}\in
\mathcal{U}$ and $\bar{X}$, $\tilde{X}$ are two solutions\vspace*{1pt} of
(\ref{EqSDELimitControl}) under $u$ and $\tilde{u}$,
respectively, such that $\Prb\circ\bar{X}(0)^{-1} = \tilde{\Prb
}\circ\tilde{X}(0)^{-1}$, then $\Prb\circ(\bar{X},u,W)^{-1} =
\tilde{\Prb}\circ(\tilde{X},\tilde{u},\tilde{W})^{-1}$ as
probability measures on $\mathcal{B}(\mathcal{X}\times\mathcal
{R}_{1}\times\mathcal{Y})$.
\end{defn}

Notice that here we give a process version of what can be equivalently
formulated in terms of probability measures on $\mathcal{B}(\mathcal{Z})$.
Indeed, any integrable control process $u$ corresponds to an
$\mathcal{R}_{1}$-valued random variable. On the other hand, since the
control appears linearly in (\ref{EqSDEPrelimitControl}) and
(\ref{EqSDELimitControl}), given any adapted $\mathcal{R}_{1}$-valued
random variable, one can find an integrable control process that
produces the same solution process $\bar{X}$
(cf. Sections 2 and 6 in Budhiraja, Dupuis and Fischer~\cite{budhirajaetal12}).

%re5.1 #&#
\begin{rem}
In Budhiraja, Dupuis and Fischer \cite{budhirajaetal12}, weak
uniqueness for (\ref{EqSDELimitControl}) is required to hold over the
class of all $\Theta\in\mathcal{P}(\mathcal{Z})$ that correspond to a
weak solution of (\ref{EqSDELimitControl}). This requirement is
stronger than necessary. As can be seen from the definition of the rate
function and the proof of Theorem 3.1 (and Theorem 7.1)
there,\footnote{In the notation of Budhiraja, Dupuis and Fischer
\cite{budhirajaetal12}, it follows from the proof of Lemma 5.1 there
and a version of Fatou's lemma, that if $Q$ is a limit point in the
sense of convergence in distribution of the sequence of $\mathcal
{P}(\mathcal{Z})$-valued random variables $Q^{N}$, then $\int_{\mathcal
{Z}} \int_{\mathbb{R}^{d_{1}}\times[0,T]} |y|^{2} r(\mrmdd y\times \mrmdd t)
Q(\mrmdd \phi\times \mrmdd r\times \mrmdd w) < \infty$ with probability one. As to the
rate function and the Laplace upper bound, notice that the class
$\mathcal{P}_{\infty}$ only contains measures $\Theta\in\mathcal
{P}(\mathcal{Z})$ such that $\int_{\mathcal{Z}} \int_{\mathbb
{R}^{d_{1}}\times[0,T]} |y|^{2} r(\mrmdd y\times \mrmdd t) \Theta(\mrmdd \phi\times
\mrmdd r\times \mrmdd w) < \infty$.} it suffices to have weak uniqueness for
(\ref{EqSDELimitControl}) over the class of all $\Theta\in\mathcal
{P}(\mathcal{Z})$ that correspond to a weak solution of (\ref
{EqSDELimitControl}) and are such that
\[
\int_{\mathcal{Z}} \int_{\mathbb{R}^{d_{1}}\times[0,T]} |y|^{2}
r(\mrmdd y\times \mrmdd t) \Theta(\mrmdd \phi\times \mrmdd r\times \mrmdd w) < \infty.
\]
This is equivalent to requiring weak uniqueness of solutions for
(\ref{EqSDELimitControl}) with respect to $\mathcal{U}$ as
in Definition \ref{DefWeakUnique} above.
\end{rem}

The Laplace principle given in Theorem \ref{ThSDELaplace} below is a
version of Theorem 7.1 in Budhiraja, Dupuis and Fischer \cite
{budhirajaetal12}; also cf. Theorem 3.1 and Remark 3.2 there. The
following assumptions are sufficient for the Laplace principle to hold:
\begin{longlist}[(H4)]
\item[(H1)] %\label{HCoeffContinuity}
The functions $b(t,\cdot,\cdot)$, $\sigma(t,\cdot,\cdot)$ are
uniformly continuous and bounded on sets $B\times P$ whenever $B
\subset\mathcal{X}$ is bounded and $P \subset\mathcal{P}(\mathbb
{R}^{d})$ is compact, uniformly in $t\in[0,T]$.

\item[(H2)] %\label{HPrelimitEq}
For all $N \in\mathbb{N}$, existence and uniqueness of solutions
holds in the strong sense for the system of $N$ equations given by
(\ref{EqSDEPrelimit}).

\item[(H3)] %\label{HLimitEq}
Weak uniqueness of solutions holds for (\ref{EqSDELimitControl}).

\item[(H4)] %\label{HTightness}
If $u^{N} \in\mathcal{U}_{N}$, $N\in\mathbb{N}$, are such that
\[
\sup_{N\in\mathbb{N}} \Mean\Biggl[ \frac{1}{N} \sum
_{i=1}^{N}\int_{0}^{T}\bigl|u_{i}^{N}(t)\bigr|^{2}
\mrmd t \Biggr] <\infty,
\]
then $\{\bar{\mu}^{N}\dvt N\in\mathbb{N}\}$ is tight as a family of
$\mathcal{P}(\mathcal{X})$-valued random variables, where $\bar{\mu
}^{N}$ is the empirical measure of the solution to the system of
equations (\ref{EqSDEPrelimitControl}) under $u^{N}$.
\end{longlist}

%th5.1 #&#
\begin{thrm}[(Budhiraja, Dupuis and Fischer \cite{budhirajaetal12})]
\label{ThSDELaplace}
Grant \textup{(H1)--(H4)}. Then the sequence $(\mu^{N})_{N\in\mathbb
{N}}$ of $\mathcal{P}(\mathcal{X})$-valued random variables satisfies the
Laplace principle with rate function $I\dvtx\mathcal{P}(\mathcal{X})
\rightarrow[0,\infty]$ given by
\[
I(\theta) = \inf_{u\in\mathcal{U}: \Law(\bar{X}^{u}) = \theta} \Mean
\biggl[ \frac{1}{2}\int
_{0}^{T} \bigl|u(t)\bigr|^{2}\mrmd t \biggr],
\]
where $\bar{X}^{u}$ is a solution of (\ref{EqSDELimitControl}) over the
time interval $[0,T]$ with $\Law(X(0)) = \delta_{x_{0}}$, and
$\inf\varnothing= \infty$ by convention.
\end{thrm}

%re5.2 #&#
\begin{rem} \label{RemGoodRateFnct}
The function $I$ of Theorem \ref{ThSDELaplace} is indeed a rate
function, that is, $I$ is lower semicontinuous with values in
$[0,\infty]$. The following hypothesis, which is analogous to the
stability condition \textup{(H4)}, is sufficient to guarantee
goodness of the rate function.
\begin{longlist}[(H$'$)]
\item[(H$'$)] If $(u_{n})_{n\in\mathbb{N}} \subset\mathcal{U}$ is
such that $\sup_{n\in\mathbb{N}} \Mean_{n} [ \int
_{0}^{T}|u_{n}(t)|^{2}\mrmd t ] <\infty$, then $\{\Law(\bar
{X}^{u_{n}})\dvt n\in\mathbb{N}\}$ is tight in $\mathcal{P}(\mathcal{X})$.
\end{longlist}
Under this additional assumption, $I$ is a good rate function and the
Laplace principle implies the large deviation principle.
\end{rem}

Consider the special case in which $d = d_{1}$, $x_{0} = 0$, $b\equiv
0$, and $\sigma\equiv\Id_{d}$. In this case, $\mathcal{X} =
\mathcal{Y}$ and $\mu^{N}$ is the empirical measure of $N$
independent Wiener processes $W^{N}_{1},\ldots,W^{N}_{N}$. Let $\gamma
_{0}$ be Wiener measure on $\mathcal{B}(\mathcal{Y})$. Since $\Law
(W^{N}_{i}) = \gamma_{0}$, Sanov's theorem implies that the sequence
$(\mu^{N})_{N\in\mathbb{N}}$ satisfies the large deviation /
Laplace principle with good rate function $R(\cdot\| \gamma_{0})$. On the
other hand, by Theorem \ref{ThSDELaplace}, $(\mu^{N})_{N\in\mathbb
{N}}$ satisfies the Laplace principle with rate function
\[
J(\gamma)\doteq\inf_{u\in\mathcal{U}: \Law(\bar{Y}^{u}) = \gamma
} \Mean\biggl[ \frac{1}{2}\int
_{0}^{T} \bigl|u(t)\bigr|^{2}\mrmd t \biggr],\qquad \gamma\in
\mathcal{P}(\mathcal{Y}),
\]
where $\bar{Y}^{u}$ is the process given by
%
%e5.5 #&#
\begin{equation}
\label{ExWienerDrift} \bar{Y}^{u}(t) \doteq\int_{0}^{t}u(s)\mrmd s
+ W(t),\qquad t\in[0,T].
\end{equation}
One checks that $J\dvtx\mathcal{P}(\mathcal{Y}) \rightarrow[0,\infty]$
has compact sublevel sets, hence is a good rate function. It follows
that $J$ coincides with the rate function obtained from Sanov's
theorem. Consequently, for all $\gamma\in\mathcal{P}(\mathcal{Y})$,
%
%e5.6 #&#
\begin{equation}
\label{EqREWiener} R ( \gamma\| \gamma_{0} ) = \inf
_{u\in\mathcal{U}: \Law
(\bar{Y}^{u}) = \gamma} \Mean\biggl[ \frac{1}{2}\int_{0}^{T}
\bigl|u(t)\bigr|^{2}\mrmd t \biggr].
\end{equation}

%re5.3 #&#
\begin{rem}
Equation (\ref{EqREWiener}) provides a ``weak'' variational
representation of relative entropy with respect to Wiener measure. In
Appendix \ref{AppREWiener}, we give a direct proof of
(\ref{EqREWiener}). The variational representation is weak in the sense
that the underlying stochastic basis may vary. In particular, the
control process $u$ may be adapted to a filtration that is strictly
bigger than the natural filtration of the Wiener process. Notice that
expectation in (\ref{EqREWiener}) is taken with respect to the
probability measure of the stochastic basis that comes with the control
process $u$.
\end{rem}

%re5.4 #&#
\begin{rem}
Representation (\ref{EqREWiener}) may be compared to the following
result obtained by {\"U}st{\"u}nel~\cite{uestuenel09}. Take as
stochastic basis the canonical set-up; in our notation, $((\mathcal
{Y},\mathcal{B}(\mathcal{Y}),\gamma_{0}),(\mathcal{B}_{t}))$, where
$(\mathcal{B}_{t})$ is the canonical filtration. Let $W$ be the
coordinate process. Thus, $W$ is a $d_{1}$-dimensional Wiener process
under $\gamma_{0}$ with respect to $(\mathcal{B}_{t})$. Let $u$ be an
$\mathbb{R}^{d_{1}}$-valued $(\mathcal{B}_{t})$-progressively
measurable process such that $\Mean_{\gamma_{0}} [\int_{0}^{T}
|u(t)|^{2}\mrmd t ] < \infty$. Consider $\bar{Y}^{u} = \int_{0}^{\cdot}u(s)\mrmd s +
W(\cdot)$. Since $\bar{Y}^{u}(\cdot,\omega) = \int_{0}^{\cdot}u(s,\omega)\mrmd s + \omega
(\cdot)$ for all $\omega\in\mathcal{Y}$,
$\bar{Y}^{u}$ induces a Borel measurable mapping $\mathcal{Y}
\rightarrow\mathcal{Y}$. Set $\gamma\doteq\gamma_{0}\circ(\bar
{Y}^{u})^{-1}$. By Theorem 8 in {\"U}st{\"u}nel \cite{uestuenel09},
%
%e5.7 #&#
\begin{equation}
\label{EqREWienerIneq} R ( \gamma\| \gamma_{0} ) \leq
\Mean_{\gamma_{0}} \biggl[\frac{1}{2}\int_{0}^{T}
\bigl|u(t)\bigr|^{2}\mrmd t \biggr].
\end{equation}
Assume in addition that $u$ is such that
\[
\Mean\biggl[ \exp\biggl(-\int_{0}^{T} u(t)
\cdot \mrmdd W(t) - \frac
{1}{2}\int_{0}^{T}
\bigl|u(t)\bigr|^{2}\mrmd t \biggr) \biggr] = 1,
\]
and that, for some $\mathbb{R}^{d_{1}}$-valued $(\mathcal
{B}_{t})$-progressively measurable process $v$,
\[
\frac{\mrmdd \gamma}{\mrmdd \gamma_{0}} = \exp\biggl(-\int_{0}^{T}
v(t)\cdot \mrmdd W(t) - \frac{1}{2}\int_{0}^{T}
\bigl|v(t)\bigr|^{2}\mrmd t \biggr),\qquad \gamma_{0}\mbox{-a.s.}
\]
Theorem 7 in {\"U}st{\"u}nel \cite{uestuenel09} then states that
equality holds in (\ref{EqREWienerIneq}) if and only if $\bar{Y}^{u}$
is $\gamma_{0}$-almost surely invertible as a mapping $\mathcal{Y}
\rightarrow\mathcal{Y}$ with inverse $\bar{Y}^{v} = \int_{0}^{\cdot}v(s)\mrmd s
+ W(\cdot)$. For similar results on abstract Wiener spaces
see Lassalle \cite{lassalle12}; Corollary 8 and Remark 4 in Section 7
therein might be compared to Lemma \ref{LemmaREWiener} in
Appendix \ref{AppREWiener} here.
\end{rem}

Let us return to the general case. Given $\theta\in\mathcal
{P}(\mathcal{X})$, denote by $\theta(t)$ the marginal distribution of
$\theta$ at time $t$ and consider the stochastic differential equation
%
%e5.8 #&#
\begin{equation}
\label{EqSDEParamLimit} \mrmdd X(t) = b \bigl(t,X,\theta(t) \bigr)\mrmd t + \sigma
\bigl(t,X,\theta(t) \bigr)\mrmd W(t).
\end{equation}
Equation (\ref{EqSDEParamLimit}) results from freezing the measure
variable in (\ref{EqSDELimit}) at $\theta$. We will assume
existence and pathwise uniqueness for (\ref{EqSDEParamLimit}).
\begin{longlist}[(H5)]
\item[(H5)] %\label{HParamLimitEq}
Given any $\theta\in\mathcal{P}(\mathcal{X})$, weak existence and
pathwise uniqueness hold for (\ref{EqSDEParamLimit}).
\end{longlist}

Based on representation (\ref{EqREWiener}) and the contraction
property of relative entropy, the rate function of Theorem \ref
{ThSDELaplace} can be shown to be of relative entropy form.

%th5.2 #&#
\begin{thrm} \label{ThSDERateFnct}
Grant \textup{(H1)--(H5)}. Then the rate function $I$ of Theorem \ref
{ThSDELaplace} can be expressed in relative entropy form as
\[
I(\theta) = R \bigl(\theta\| \Psi(\theta) \bigr),\qquad \theta\in\mathcal{P}(
\mathcal{X}),
\]
where $\Psi(\theta)$ is the law of the unique solution of
(\ref{EqSDEParamLimit}) under $\theta$ over the time
interval $[0,T]$ with initial condition $X(0) = x_{0}$.
\end{thrm}

%re5.5 #&#
\begin{rem} The hypotheses of Theorem \ref{ThSDERateFnct} are
satisfied if $b$, $\sigma$ are locally Lipschitz continuous with
$\sigma$ uniformly bounded and $b$ of sub-linear growth in the
trajectory variable; see Appendix~\ref{AppItoSuff}. These sufficient
conditions are at the same time more restrictive and more general than
the assumptions made in Dawson and G{\"a}rtner \cite
{dawsongaertner87}, where the large deviation principle is derived for
weakly interacting It{\^o} diffusions. There the coefficients are only
required to be continuous, where continuity in the measure variable is
with respect to an inductive topology that is stronger than the
topology of weak convergence (but cf. Remark \ref{RemMeasureTopology}
above), and to satisfy a coercivity condition that allows for
sub-linear growth of the dispersion coefficient and for super-linear
growth of the drift vector in ``stabilizing'' directions. On the other
hand, in Dawson and G{\"a}rtner \cite{dawsongaertner87} the diffusion
matrix has to be non-degenerate and independent of the measure
variable, while here we can have degeneracy of $\sigma{\sigma
}^\mathsf{T}$ as well as measure dependence. Lastly, since here both
$b$ and
$\sigma$ are functions of the entire trajectory history, one can
capture systems with delay in the state dynamics.
\end{rem}

%re5.6 #&#
\begin{rem}
Assumption \textup{(H5)} can be weakened by requiring weak existence
and pathwise uniqueness of solutions to (\ref
{EqSDEParamLimit}) only for $\theta\in\mathcal{P}(\mathcal{X})$ such
that $I(\theta) < \infty$. Those measures $\theta$ are, by
definition of $I$, distributions of It{\^o} processes. The function
$\Psi$ introduced in Theorem \ref{ThSDERateFnct} would then be
defined only on the effective domain of $I$; for $\theta\in\mathcal
{P}(\mathcal{X})$ with $I(\theta) = \infty$, one can then choose
$\Psi(\theta)$ in such a way that $\theta$ is not absolutely
continuous with respect to $\Psi(\theta)$ (e.g., by choosing
between two Dirac measures).
\end{rem}

\begin{pf*}{Proof of Theorem \ref{ThSDERateFnct}}
Let $\theta\in\mathcal{P}(\mathcal{X})$. By hypothesis, weak
existence and pathwise uniqueness hold for
(\ref{EqSDEParamLimit}). By a result originally due to
Yamada and Watanabe \cite{yamadawatanabe71} (also cf. Kallenberg \cite
{kallenberg96}), there is a
Borel measurable mapping $\psi_{\theta}\dvtx\mathbb{R}^{d}\times
\mathcal{Y} \rightarrow\mathcal{X}$ such that
%
%e5.9 #&#
\begin{equation}
\label{EqYamadaWatanabe} \psi_{\theta}(x_{0},W) = X,\qquad \Prb
\mbox{-almost surely},
\end{equation}
whenever $X$ is a solution of (\ref{EqSDEParamLimit}) under
$\theta$ over time $[0,T]$ with initial condition $X(0) = x_{0}$ on
some stochastic basis $((\Omega,\mathcal{F},\Prb),(\mathcal
{F}_{t}))$ carrying a $d_{1}$-dimensional Wiener process $W$. For such
a solution, $\Psi(\theta) = \Law(X)$ by definition. Set $\psi
_{\theta}(\cdot)\doteq\psi_{\theta}(x_{0},\cdot)$, and let $\gamma
_{0}$ be Wiener measure on~$\mathcal{B}(\mathcal{Y})$. By
(\ref
{EqYamadaWatanabe}), $\Psi(\theta) = \psi_{\theta}(\gamma_{0}) =
\gamma_{0}\circ\psi_{\theta}^{-1}$. By Lemma \ref{LemmaRE}, the
contraction property of relative entropy, and representation (\ref
{EqREWiener}) it follows that
\begin{eqnarray*}
R \bigl(\theta\| \Psi(\theta) \bigr) &=& R \bigl(\theta\| \psi_{\theta}(
\gamma_{0}) \bigr)
\\
&=& \inf_{\gamma\in\mathcal{P}(\mathcal{Y}): \psi_{\theta}(\gamma
) =
\theta} R ( \gamma\| \gamma_{0} )
\\
&=& \inf_{\gamma\in\mathcal{P}(\mathcal{Y}): \psi_{\theta}(\gamma
) =
\theta} \inf_{u\in\mathcal{U}: \Law(\bar{Y}^{u}) = \gamma} \Mean\biggl
[
\frac{1}{2} \int_{0}^{T}
\bigl|u(t)\bigr|^{2}\mrmd t \biggr]
\\
&=& \inf_{u\in\mathcal{U}: \Law(\psi_{\theta}(\bar{Y}^{u})) =
\theta} \Mean\biggl[ \frac{1}{2} \int
_{0}^{T} \bigl|u(t)\bigr|^{2}\mrmd t \biggr],
\end{eqnarray*}
where $\bar{Y}^{u}$ is defined by (\ref{ExWienerDrift}). Let $u\in
\mathcal{U}$, and set $\tilde{X}^{u}\doteq\psi_{\theta}(\bar
{Y}^{u})$. Then, as a consequence of (\ref
{EqYamadaWatanabe}), $\tilde{X}^{u}$ solves
\[
\mrmdd X(t) = b \bigl(t,X,\theta(t) \bigr)\mrmd t + \sigma\bigl(t,X,\theta(t)
\bigr)u(t)\mrmd t + \sigma\bigl(t,X,\theta(t) \bigr)\mrmd W(t)
\]
with initial distribution $\delta_{x_{0}}$. If $u$ is such that $\Law
(\psi_{\theta}(\bar{Y}^{u})) = \theta$, then $\tilde{X}^{u}$ is a
solution of (\ref{EqSDELimitControl}) under $u$ with initial
distribution $\delta_{x_{0}}$. By assumption \textup{(H3)}, weak
uniqueness holds for (\ref{EqSDELimitControl}), hence $\Law
(\tilde{X}^{u}) = \Law(\bar{X}^{u})$ whenever $\bar{X}^{u}$ is a
solution of (\ref{EqSDELimitControl}) under $u$ with $\Law(\bar
{X}^{u}(0)) = \delta_{x_{0}}$. It follows that
\begin{eqnarray*}
R \bigl(\theta\| \Psi(\theta) \bigr) &=& \inf_{u\in\mathcal{U}:
\Law(\psi_{\theta}(\bar{Y}^{u})) = \theta} \Mean
\biggl[ \frac
{1}{2} \int_{0}^{T}
\bigl|u(t)\bigr|^{2}\mrmd t \biggr]
\\
&=& \inf_{u\in\mathcal{U}: \Law(\bar{X}^{u}) = \theta} \Mean\biggl[
\frac{1}{2} \int
_{0}^{T} \bigl|u(t)\bigr|^{2}\mrmd t \biggr]
\\
&=& I(\theta),
\end{eqnarray*}
where $I$ is the rate function of Theorem \ref{ThSDELaplace}.
\end{pf*}

%re5.7 #&#
\begin{rem}
Assuming in addition to \textup{(H1)--(H4)} hypothesis \textup{(H$'$)} of
Remark \ref{RemGoodRateFnct}, Theorem \ref{ThSDERateFnct} can be
proved by applying both Sanov's theorem and Theorem \ref{ThSDELaplace}
to the weakly interacting system given by equations (\ref
{EqSDEPrelimit}) with measure variable frozen at $\theta\in\mathcal
{P}(\mathcal{X})$ and then evaluating the resulting rate functions at
$\theta$.
\end{rem}

%sA #&#
\begin{appendix}\label{app}

%sA #&#
\section{Contraction property of relative entropy} \label{AppREContraction}

Let $\mathcal{X}$, $\mathcal{Y}$ be Polish spaces. Denote by $\Pi
_{\mathcal{X}}$, $\Pi_{\mathcal{Y}}$ the collection of all finite
and measurable partitions of $\mathcal{X}$ and $\mathcal{Y}$,
respectively. Recall that relative entropy can be approximated in terms
of finite sums; for $\eta, \nu\in\mathcal{P}(\mathcal{X})$,
%
%eA.1 #&#
\begin{equation}
\label{EqREApprox} R(\eta\|\nu) = \sup_{\pi\in\Pi_{\mathcal{X}}} \sum
_{A\in\pi} \eta(A)\log\biggl(\frac{\eta(A)}{\nu(A)} \biggr),
\end{equation}
see, for instance, Lemma 1.4.3(g) in Dupuis and Ellis \cite
{dupuisellis97}, page 30. For $\psi\dvtx\mathcal{Y} \rightarrow
\mathcal{X}$ measurable, $\gamma\in\mathcal{P}(\mathcal{Y})$,
denote by $\psi(\gamma) \doteq\gamma\circ\psi^{-1}$ the image
measure of $\gamma$ under $\psi$.

The following lemma extends the invariance property of relative entropy
under bijective bi-measurable mappings as given by Lemma E.2.1 in
Dupuis and Ellis \cite{dupuisellis97}, page 366, to arbitrary
measurable transformations; also cf. Theorem 2.4.1 in
Kullback \cite{kullback78}, pages 19 and 20, where the inequality that
is implied by Lemma \ref{LemmaRE} is established.

%leA.1 #&#
\begin{lemma} \label{LemmaRE}
Let $\psi\dvtx\mathcal{Y} \rightarrow\mathcal{X}$ be a Borel
measurable mapping. Let $\eta\in\mathcal{P}(\mathcal{X})$, $\gamma
_{0} \in\mathcal{P}(\mathcal{Y})$. Then
%
%eA.2 #&#
\begin{equation}
\label{EqLemmaRE} R \bigl(\eta\| \psi(\gamma_{0}) \bigr) = \inf
_{\gamma\in\mathcal
{P}(\mathcal{Y}): \psi(\gamma) = \eta} R (\gamma\| \gamma_{0} ),
\end{equation}
where $\inf\varnothing= \infty$ by convention.
\end{lemma}

\begin{pf}
Suppose $\gamma\in\mathcal{P}(\mathcal{Y})$ is such that $\psi
(\gamma) = \eta$. Then, by (\ref{EqREApprox}) and the definition of
image measure,
\begin{eqnarray*}
R \bigl(\eta\| \psi(\gamma_{0}) \bigr) &=& \sup_{\pi\in\Pi
_{\mathcal{X}}}
\sum_{A\in\pi} \eta(A)\log\biggl(\frac{\eta
(A)}{\psi(\gamma_{0})(A)}
\biggr)
\\
&=& \sup_{\pi\in\Pi_{\mathcal{X}}} \sum_{A\in\pi} \gamma
\bigl(\psi^{-1}(A)\bigr)\log\biggl(\frac{\gamma(\psi^{-1}(A))}{\gamma
_{0}(\psi
^{-1}(A))} \biggr)
\\
&=& \sup_{\pi\in\Pi_{\mathcal{X}}} \sum_{B\in\psi^{-1}(\pi)}
\gamma(B)\log\biggl(\frac{\gamma(B)}{\gamma_{0}(B)} \biggr)
\\
&\leq&\sup_{\hat{\pi}\in\Pi_{\mathcal{Y}}} \sum_{B\in\hat{\pi
}}
\gamma(B)\log\biggl(\frac{\gamma(B)}{\gamma_{0}(B)} \biggr)
\\
&=& R (\gamma\| \gamma_{0} ),
\end{eqnarray*}
where $\psi^{-1}(\pi)$ denotes the partition of $\mathcal{Y}$
induced by the inverse images of $\psi$. More precisely, $\psi
^{-1}(\pi) \doteq\{ \psi^{-1}(A)\dvt A \in\pi\}$. Notice that $\psi
^{-1}(\pi)$ is indeed a finite and measurable partition of $\mathcal
{Y}$ since $\pi$ is a finite and measurable partition of $\mathcal
{X}$, inverse images under $\psi$ are Borel measurable and $\psi
^{-1}(A)\cap\psi^{-1}(\tilde{A}) = \varnothing$ whenever $A\cap
\tilde{A} = \varnothing$. Since $\inf\varnothing= \infty$, it follows that
\[
R \bigl(\eta\| \psi(\gamma_{0}) \bigr) \leq\inf_{\gamma\in
\mathcal{P}(\mathcal{Y}): \psi(\gamma) = \eta}
R (\gamma\| \gamma_{0} ).
\]
If $R (\eta\| \psi(\gamma_{0}) ) = \infty$, then the
above inequality is necessarily an equality, namely $\infty= \infty$.
Thus in order to show the opposite inequality, we may assume that
$R (\eta\| \psi(\gamma_{0}) ) < \infty$. Now $R
(\eta\| \psi(\gamma_{0}) ) < \infty$ implies that $\eta$ is
absolutely continuous with respect to $\psi(\gamma_{0})$, hence
possesses a density $f\doteq\frac{\mrmdd \eta}{\mrmdd \psi(\gamma_{0})}$. Set
\[
\gamma(C) \doteq\int_{C} f\bigl(\psi(y)\bigr)
\gamma_{0}(\mrmdd y),\qquad C\in\mathcal{B}(\mathcal{Y}).
\]
Then $\gamma$ is a probability measure %since the integrand is
%nonnegative and, by the integral transformation formula and definition
%of $f$,
% \gamma(\mathcal{Y}) = \int_{\mathcal{Y}} f(\psi(y))\gamma_{0}(\mrmdd y) =
%Thus $\gamma$
having density $f\circ\psi$ with respect to $\gamma_{0}$. Using the
integral transformation formula and definition of $f$, we have for all
$A\in\mathcal{B}(\mathcal{X})$,
\begin{eqnarray*}
\psi(\gamma) (A) &=& \int_{\mathcal{Y}} \mathbf{1}_{\psi
^{-1}(A)}(y)
\cdot f\bigl(\psi(y)\bigr) \gamma_{0}(\mrmdd y)
\\
&=& \int_{\mathcal{Y}} \mathbf{1}_{A}\bigl(\psi(y)\bigr)
\cdot f\bigl(\psi(y)\bigr) \gamma_{0}(\mrmdd y)
\\
&=& \int_{\mathcal{X}} \mathbf{1}_{A}(x)\cdot f(x) \psi(
\gamma_{0}) (\mrmdd x)
\\
&=& \eta(A),
\end{eqnarray*}
which means that $\psi(\gamma) = \eta$. Recalling that $f\circ\psi
= \frac{\mrmdd \gamma}{\mrmdd \gamma_{0}}$, $f = \frac{\mrmdd \eta}{\mrmdd \psi(\gamma_{0})}$,
\begin{eqnarray*}
R (\gamma\| \gamma_{0} ) &=& \int_{\mathcal{Y}} f\bigl(
\psi(y)\bigr) \log\bigl(f\bigl(\psi(y)\bigr) \bigr) \gamma_{0}(\mrmdd y)
\\
&=& \int_{\mathcal{X}} f(x) \log\bigl(f(x) \bigr) \psi(\gamma
_{0}) (\mrmdd x)
\\
&=& R \bigl(\eta\| \psi(\gamma_{0}) \bigr),
\end{eqnarray*}
which proves inequality ``$\geq$'' in (\ref{EqLemmaRE}).
\end{pf}

The proof of Lemma \ref{LemmaRE} shows that the probability measure
$\gamma$ defined by $\gamma(\mrmdd y) \doteq\frac{\mrmdd \eta}{\mrmdd \psi(\gamma
_{0})}(\psi(y)) \gamma_{0}(\mrmdd y)$ attains the infimum in (\ref
{EqLemmaRE}) whenever that infimum is finite.

%sA #&#
\section{Relative entropy with respect to Wiener measure} \label{AppREWiener}

Let $\mathcal{Y}$ be the Polish space $\mathbf{C}([0,T],\mathbb
{R}^{d})$ equipped with the maximum norm topology. Let $\mathcal{U}$
be defined as in Section \ref{SectIto} with $d_{1} = d$. Thus,
$\mathcal{U}$ is the set of quadruples $((\Omega,\mathcal{F},\Prb
),(\mathcal{F}_{t}),u,W)$ such that the pair $((\Omega,\mathcal
{F},\Prb),(\mathcal{F}_{t}))$ forms a stochastic basis satisfying the
usual hypotheses, $W$ is a $d$-dimensional $(\mathcal{F}_{t})$-Wiener
process, and $u$ is an $\mathbb{R}^{d}$-valued $(\mathcal
{F}_{t})$-progressively measurable process with $\Mean[ \int_{0}^{T}
|u(t)|^{2}\mrmd t ] < \infty$. Given $u\in\mathcal{U}$,
define $\bar{Y}^{u}$ according to (\ref{ExWienerDrift}), that is,
\[
\bar{Y}^{u}(t)\doteq W(t) + \int_{0}^{t}
u(s)\mrmd s,\qquad t\in[0,T].
\]
The following result provides a variational representation of relative
entropy with respect to Wiener measure.

%leA.1 #&#
\begin{lemma} \label{LemmaREWiener}
Let $\gamma_{0}$ be Wiener measure on $\mathcal{B}(\mathcal{Y})$. Then
for all $\gamma\in\mathcal{P}(\mathcal{Y})$,
%
%eA.1 #&#
\begin{equation}
\label{EqLemmaREWiener} R (\gamma\| \gamma_{0} ) = \inf
_{u\in\mathcal{U}: \Law
(\bar{Y}^{u}) = \gamma} \Mean\biggl[\frac{1}{2} \int_{0}^{T}
\bigl|u(t)\bigr|^{2}\mrmd t \biggr],
\end{equation}
where $\inf\varnothing= \infty$ by convention.
\end{lemma}

The proof of inequality ``$\leq$'' in (\ref{EqLemmaREWiener}) relies
on the lower semicontinuity of relative entropy and the
Donsker--Varadhan variational formula; it may be confronted to the first
part of the proof of Theorem 3.1 in Bou{\'e} and Dupuis \cite
{bouedupuis98}. The proof of inequality ``$\geq$'' exploits the
variational formulation and uses arguments contained in
F{\"o}llmer \cite{foellmer85,foellmer86}.

\begin{pf*}{Proof of Lemma \ref{LemmaREWiener}}
In order to prove inequality ``$\leq$'' in (\ref{EqLemmaREWiener}),
it suffices to show that, for all $u\in\mathcal{U}$,
%
%eA.2 #&#
\begin{equation}
\label{EqLemmaREWienerLeq} R \bigl(\Law\bigl(\bar{Y}^{u}\bigr) \|
\gamma_{0} \bigr) \leq\Mean\biggl[\frac{1}{2} \int
_{0}^{T} \bigl|u(t)\bigr|^{2}\mrmd t \biggr].
\end{equation}
Let $u\in\mathcal{U}$, and set $\gamma\doteq\Law(\bar{Y}^{u}) =
\Prb\circ(\bar{Y}^{u})^{-1}$. In accordance with Definition 3.2.3 in
Karatzas and Shreve \cite{karatzasshreve91},
page 132, a process $v$ defined on $((\Omega,\mathcal{F},\Prb
),(\mathcal{F}_{t}))$ is called simple if there are $N\in\mathbb
{N}$, $0=t_{0}<\cdots<t_{N}=T$, and uniformly bounded $\mathbb
{R}^{d}$-valued random variables $\xi_{0},\ldots,\xi_{N}$ such that
$\xi_{i}$ is $\mathcal{F}_{t_{i}}$-measurable and
\[
v(t,\omega) = \xi_{0}(\omega) \mathbf{1}_{\{0\}}(t) +
\sum_{i=0}^{N} \xi_{i}(\omega)
\mathbf{1}_{(t_{i},t_{i+1}]}(t).
\]
By Proposition 3.2.6 in Karatzas and Shreve \cite{karatzasshreve91},
page 134, there exists a sequence $(v_{n})_{n\in\mathbb{N}}$ of
simple processes such that $\Mean[\int_{0}^{T}
|u(t)-v_{n}(t)|^{2}\mrmd t ] \to0$ as $n\to\infty$. Let
$(v_{n})_{n\in\mathbb{N}}$ be such a sequence. For $n\in\mathbb
{N}$, set $\gamma_{n}\doteq\Law(\bar{Y}^{v_{n}})$. Then $\gamma
_{n} \to\gamma$ in $\mathcal{P}(\mathcal{Y})$ since
\[
\Mean\Bigl[\sup_{t\in[0,T]} \bigl|\bar{Y}^{u}(t)- \bar
{Y}^{v_{n}}(t)\bigr|^{2} \Bigr] \leq T\cdot\Mean\biggl[\int
_{0}^{T} \bigl|u(t)-v_{n}(t)\bigr|^{2}\mrmd t
\biggr] \stackrel{n\to\infty} {\longrightarrow} 0.
\]
Therefore, by the lower semicontinuity of $R(\cdot\|\gamma_{0})$,
\begin{eqnarray*}
R \bigl(\Law\bigl(\bar{Y}^{u}\bigr) \| \gamma_{0} \bigr) &=&
R (\gamma\| \gamma_{0} ) \leq\liminf_{n\to\infty} R (
\gamma_{n} \| \gamma_{0} ) \\
&=& \liminf_{n\to\infty}
R \bigl(\Law\bigl(\bar{Y}^{v_{n}}\bigr) \| \gamma_{0} \bigr).
\end{eqnarray*}
On the other hand, $\Mean[\frac{1}{2}\int_{0}^{T}
|v_{n}(t)|^{2}\mrmd t ] \to\Mean[\frac{1}{2}\int_{0}^{T}
|u(t)|^{2}\mrmd t ]$ as $n\to\infty$. It is therefore enough to
show that (\ref{EqLemmaREWienerLeq}) holds whenever $u$ is a simple
process. Thus, assume that $u$ is simple. Let $Z$ be the $\mathcal
{F}_{T}$-measurable $(0,\infty)$-valued random variable given by
\[
Z\doteq\exp\biggl(-\int_{0}^{T} u(s)\cdot \mrmdd W(s)
- \frac{1}{2} \int_{0}^{T}
\bigl|u(s)\bigr|^{2}\mrmd s \biggr).
\]
Notice that $\Mean[ Z] = 1$ since $u$ is uniformly bounded. Define a
probability measure $\tilde{\Prb}$ on $(\Omega,\mathcal{F}_{T})$ by
\[
\frac{\mrmdd \tilde{\Prb}}{\mrmdd \Prb} \doteq Z.
\]
By Girsanov's theorem (Theorem 3.5.1 in Karatzas and Shreve \cite
{karatzasshreve91}, page
191), $\bar{Y}^{u}$ is an $(\mathcal{F}_{t})$-Wiener process with
respect to $\tilde{\Prb}$. By the Donsker--Varadhan variational formula
for relative entropy (Lemma 1.4.3(a) in Dupuis and Ellis \cite
{dupuisellis97}, page 29),
%
%eA.3 #&#
\begin{equation}
\label{EqDonskerVaradhan} R (\gamma\| \gamma_{0} ) = \sup
_{g\in\mathbf{C}_{b}(\mathcal{Y})} \biggl\{\int_{\mathcal{Y}} g(y) \gamma(\mrmdd y) -
\log\int_{\mathcal{Y}} \RMe^{g(y)} \gamma_{0}(\mrmdd y)
\biggr\}.
\end{equation}
Recall that $\gamma= \Prb\circ(\bar{Y}^{u})^{-1}$ and $\gamma_{0}
= \Prb\circ W^{-1}$, but also $\gamma_{0} = \tilde{\Prb}\circ(\bar
{Y}^{u})^{-1}$ since $\bar{Y}^{u}$ is a Wiener process under $\tilde
{P}$. Let $\tilde{\Mean}$ denote expectation with respect to $\tilde
{\Prb}$. By the convexity of $-\log$ and Jensen's inequality, for all
$g\in\mathbf{C}_{b}(\mathcal{Y})$,
\begin{eqnarray*}
&&\int_{\mathcal{Y}} g(y) \gamma(\mrmdd y) - \log\int_{\mathcal{Y}}
\RMe^{g(y)} \gamma_{0}(\mrmdd y)
\\
&&\quad= \Mean\bigl[g\bigl(\bar{Y}^{u}\bigr) \bigr] - \log\Mean\bigl[\exp
\bigl(g(W) \bigr) \bigr]
\\
&&\quad= \Mean\bigl[g\bigl(\bar{Y}^{u}\bigr) \bigr] - \log\tilde{\Mean}
\bigl[\exp\bigl(g\bigl(\bar{Y}^{u}\bigr) \bigr) \bigr]
\\
&&\quad= \Mean\bigl[g\bigl(\bar{Y}^{u}\bigr) \bigr] - \log\Mean\bigl[\exp
\bigl(g\bigl(\bar{Y}^{u}\bigr) \bigr)\cdot Z \bigr]
\\
&&\quad= \Mean\bigl[g\bigl(\bar{Y}^{u}\bigr) \bigr]\\
&&\qquad{} - \log\Mean\biggl[
\exp\biggl(g\bigl(\bar{Y}^{u}\bigr) - \int_{0}^{T}
u(t)\cdot \mrmdd W(t) - \frac{1}{2}\int_{0}^{T}
\bigl|u(t)\bigr|^{2}\mrmd t \biggr) \biggr]
\\
&&\quad\leq\Mean\bigl[g\bigl(\bar{Y}^{u}\bigr) \bigr] - \Mean\biggl[g
\bigl(\bar{Y}^{u}\bigr) - \int_{0}^{T}
u(t)\cdot \mrmdd W(t) - \frac{1}{2}\int_{0}^{T}
\bigl|u(t)\bigr|^{2}\mrmd t \biggr]
\\
&&\quad= \Mean\biggl[\frac{1}{2}\int_{0}^{T}
\bigl|u(t)\bigr|^{2}\mrmd t \biggr]
\end{eqnarray*}
since $\Mean[\int_{0}^{T} u(t)\cdot \mrmdd W(t) ] = 0$ as $u$
is square integrable. In view of (\ref{EqDonskerVaradhan}), inequality
(\ref{EqLemmaREWienerLeq}) follows.

In order to prove inequality ``$\geq$'' in (\ref{EqLemmaREWiener}),
it suffices to consider probability measures with finite relative
entropy with respect to Wiener measure. Let $\gamma\in\mathcal
{P}(\mathcal{Y})$ be such that $R(\gamma\| \gamma_{0}) < \infty$. In
particular, $\gamma$ is absolutely continuous with respect to $\gamma
_{0}$. We have to show that there exists $u\in\mathcal{U}$ such that
$\Law(\bar{Y}^{u}) = \gamma$ and $R(\gamma\| \gamma_{0}) \geq
\Mean[\frac{1}{2}\int_{0}^{T} |u(t)|^{2}\mrmd t ]$. Let $Y$
be the coordinate process on the canonical space $(\mathcal
{Y},\mathcal{B}(\mathcal{Y}))$, and let $(\mathcal{B}_{t})_{t\in
[0,T]}$ be the
canonical filtration (the natural filtration of $Y$). Denote by $(\hat
{\mathcal{B}}_{t})$ the $\gamma_{0}$-augmentation of $(\mathcal
{B}_{t})$. Both $\gamma_{0}$ and $\gamma$ extend naturally to $\hat
{\mathcal{B}}_{T} \supset\mathcal{B}(\mathcal{Y})$. Clearly, $Y$ is a
$(\hat{\mathcal{B}}_{t})$-Wiener process under $\gamma_{0}$. Since
$R(\gamma\| \gamma_{0}) < \infty$, there is a $[0,\infty)$-valued
$\hat{\mathcal{B}}_{T}$-measurable random variable $\xi$ such that
\begin{eqnarray*}
\frac{\mrmdd \gamma}{\mrmdd \gamma_{0}} &=& \xi,\\
\Mean_{\gamma_{0}} [ \xi] &=& 1,\\
\Mean
_{\gamma}
\bigl[ \bigl|\log(\xi)\bigr| \bigr] &=& \Mean_{\gamma_{0}} \bigl[ \bigl|\log(\xi)\bigr|\xi
\bigr] <
\infty.
\end{eqnarray*}
Set $Z(t)\doteq\Mean_{\gamma_{0}}[\xi| \hat{\mathcal{B}}_{t} ]$,
$t\in[0,T]$. %Then $Z$ is a martingale under $\gamma_{0}$, which we
%may assume to be right-continuous [Theorem 1.3.13 in
By a version of It{\^o}'s martingale representation theorem
(Theorem III.4.33 in Jacod and Shiryaev \cite{jacodshiryaev03}, page
189), there exists an $\mathbb{R}^{d}$-valued $(\hat{\mathcal
{B}}_{t})$-progressively measurable process $v$ such that $\gamma
_{0}(\int_{0}^{T} |v(t)|^{2}\mrmd t < \infty) = 1$ and
%
%eA.4 #&#
\begin{equation}
\label{EqLemmaREWienerRep} Z(t) = 1 + \int_{0}^{t}
v(s)\cdot \mrmdd Y(s) \qquad\mbox{for all }t\in[0,T],\qquad
\gamma_{0}\mbox{-a.s.}
\end{equation}
In particular, $Z$ is a continuous process. By the continuity and
martingale property of $Z$, and since $Z(T) = \xi$,
\[
\gamma\Bigl( \inf_{t\in[0,T]} Z(t) > 0 \Bigr) = 1.
\]
%
%The above event need not have probability one under $\gamma_{0}$.
Define an $\mathbb{R}^{d}$-valued $(\hat{\mathcal
{B}}_{t})$-progressively measurable process $u$ by
%
%eA.5 #&#
\begin{equation}
\label{ExLemmaREWienerControl} u(t) \doteq\frac{1}{Z(t)}\cdot v(t) \cdot
\mathbf{1}_{\{\inf
_{s\in[0,t]} Z(s) >0 \}},\qquad t\in[0,T].
\end{equation}
Thus $u(t) = v(t)/Z(t)$ $\gamma$-almost surely. Applying It{\^o}'s
formula to calculate $\log(Z(t))$ (more precisely, It{\^o}'s formula
is applied to $\phi_{\epsilon}(Z(t))$ with $\phi_{\epsilon} \in
\mathbf{C}^{2}(\mathbb{R})$ such that $\phi_{\epsilon}(x) = \log
(x)$ for all $x \geq\epsilon> 0$), one checks that
%
%eA.6 #&#
\begin{equation}
\label{EqLemmaREWienerExp} Z(t) = \exp\biggl( \int_{0}^{t}u(s)
\cdot \mrmdd Y(s) - \frac{1}{2} \int_{0}^{t}
\bigl|u(s)\bigr|^{2}\mrmd t \biggr)\qquad\mbox{for all } t\in[0,T],\qquad \gamma\mbox{-a.s.}
\end{equation}
Set $\tilde{Y}(t) \doteq Y(t) - \int_{0}^{t} u(s)\mrmd s$, $t\in[0,T]$.
Then $\tilde{Y}$ is a $(\hat{\mathcal{B}}_{t})$-Wiener process with
respect to $\gamma$. Clearly, $\tilde{Y}$ is continuous and
$(\hat{\mathcal{B}}_{t})$-adapted. Since $\gamma$ is absolutely
continuous with respect to $\gamma_{0}$, the quadratic covariation
processes of $Y$ are the same with respect to $\gamma_{0}$ as with
respect to $\gamma$. Since $\int_{0}^{\cdot}u(t)\mrmd t$ is a process of finite
total variation with $\gamma$-probability one, it follows that
$\tilde{Y}$ has the same quadratic covariations under $\gamma$ as $Y$
under $\gamma_{0}$. In view of L{\'e}vy's characterization of the
Wiener process (Theorem 3.3.16 in Karatzas and Shreve \cite
{karatzasshreve91}, page 157), it
suffices to check that $\tilde{Y}$ is a local martingale with respect
to $(\hat{\mathcal{B}}_{t})$ and $\gamma$. But this follows from the
version of Girsanov's theorem provided by Theorem III.3.11 in
Jacod and Shiryaev \cite{jacodshiryaev03}, pages 168 and 169, and the
fact that, thanks to
(\ref{EqLemmaREWienerRep}), the quadratic covariations of the
continuous processes $Y_{i}$, $i\in\{1,\ldots,d\}$, and $Z$ are given
by
\[
[ Y_{i}, Z ](t) = \langle Y_{i}, Z \rangle(t) = \int
_{0}^{t} v_{i}(s)\mrmd s \qquad\mbox{for all }t
\in[0,T],\qquad \gamma_{0}\mbox{-a.s.},
\]
and $v(t) = u(t)\cdot Z(t)$ $\gamma$-almost surely. For $n\in\mathbb
{N}$, define a $(\hat{\mathcal{B}}_{t})$-stopping time $\tau_{n}$ by
\[
\tau_{n}\doteq\inf\biggl\{ t \geq0\dvt\int_{0}^{t}\bigl|u(s)\bigr|^{2}\mrmd s
> n \biggr\}\wedge T.
\]
Set
\[
\xi_{n}\doteq\exp\biggl( \int_{0}^{\tau_{n}}
u(s)\cdot \mrmdd Y(s) - \frac{1}{2} \int_{0}^{\tau_{n}}
\bigl|u(s)\bigr|^{2}\mrmd s \biggr).
\]
Then $\xi_{n}$ is well defined with $\xi_{n} > 0$ $\gamma
_{0}$-almost surely (hence also $\gamma$-almost surely). By Novikov's criterion
(Corollary 3.5.13 in Karatzas and Shreve \cite{karatzasshreve91}, page
199) and the version of Girsanov's theorem cited in the first part of
the proof,
\[
\frac{\mrmdd \gamma_{n}}{\mathrm{d}\gamma_{0}} \doteq\xi_{n}
\]
defines a probability measure $\gamma_{n}$ which is equivalent to
$\gamma_{0}$. As a consequence, $\gamma$ is absolutely continuous
with respect to $\gamma_{n}$ with density given by $\xi/\xi_{n}$. It
follows that
\begin{eqnarray*}
R ( \gamma\| \gamma_{0} ) &=& \Mean_{\gamma} \bigl[ \log(\xi)
\bigr]
\\
&=& \Mean_{\gamma} \biggl[ \log\biggl(\frac{\xi}{\xi_{n}} \biggr) \biggr
] +
\Mean_{\gamma} \bigl[ \log(\xi_{n}) \bigr]
\\
&=& R ( \gamma\| \gamma_{n} ) + \Mean_{\gamma} \biggl[ \int
_{0}^{\tau_{n}} u(s)\cdot \mrmdd Y(s) - \frac{1}{2} \int
_{0}^{\tau_{n}} \bigl|u(s)\bigr|^{2}\mrmdd s \biggr]
\\
&=& R ( \gamma\| \gamma_{n} ) + \Mean_{\gamma} \biggl[ \int
_{0}^{\tau_{n}} u(s)\cdot \mrmdd \tilde{Y}(s) \biggr] +
\Mean_{\gamma
} \biggl[ \frac{1}{2} \int_{0}^{\tau_{n}}
\bigl|u(s)\bigr|^{2}\mrmd s \biggr]
\\
&=& R ( \gamma\| \gamma_{n} ) + \Mean_{\gamma} \biggl[
\frac{1}{2} \int_{0}^{\tau_{n}}
\bigl|u(s)\bigr|^{2}\mrmd s \biggr]
\end{eqnarray*}
since $\tilde{Y}$ is a $\gamma$-Wiener process and $\int_{0}^{T}
\mathbf{1}_{[0,\tau_{n}]}(s)\cdot|u(s)|^{2}\mrmd s \leq n$ by
construction of $\tau_{n}$. Since relative entropy is non-negative and
$\Mean_{\gamma} [ \frac{1}{2} \int_{0}^{\tau_{n}}
|u(s)|^{2}\mrmd s ] \to\Mean_{\gamma} [ \frac{1}{2} \int_{0}^{T}
|u(s)|^{2}\mrmd s ]$ in $[0,\infty]$ as $n\to\infty$ by
monotone convergence, we obtain
%
%eA.7 #&#
\begin{equation}
\label{EqLemmaREWienerEst} R ( \gamma\| \gamma_{0} ) \geq
\Mean_{\gamma} \biggl[ \frac{1}{2} \int_{0}^{T}
\bigl|u(s)\bigr|^{2}\mrmd s \biggr].
\end{equation}
Since $R(\gamma\| \gamma_{0}) < \infty$ by assumption, also $\Mean
_{\gamma} [ \frac{1}{2} \int_{0}^{T} |u(s)|^{2}\mrmd s ] <
\infty$, which together with (\ref{EqLemmaREWienerExp}) actually
implies equality in (\ref{EqLemmaREWienerEst}).

Now we are in a position to choose $((\Omega,\mathcal{F},\Prb
),(\mathcal{F}_{t}),u,W) \in\mathcal{U}$ such that
\[
\Prb\circ\biggl( W + \int_{0}^{\cdot}u(s)\mrmd s
\biggr)^{-1} = \gamma\quad\mbox{and}\quad R (\gamma\| \gamma_{0} ) \geq
\Mean\biggl[ \frac
{1}{2} \int_{0}^{T}
\bigl|u(s)\bigr|^{2}\mrmd s \biggr].
\]
Take $\Omega\doteq\mathcal{Y}$, let $\mathcal{F}$ be the $\gamma
$-completion of $\mathcal{B}_{T}$, and take $\Prb$ equal to $\gamma
$, extended to the additional null sets. Let $(\mathcal{F}_{t})$ be
the $\gamma$-augmentation of $(\mathcal{B}_{t})$. Notice that $\hat
{\mathcal{B}}_{t} \subseteq\mathcal{F}_{t}$, $t\in[0,T]$, and that
$(\mathcal{F}_{t})$ satisfies the usual hypotheses. Define the control
process $u$ according to (\ref{ExLemmaREWienerControl}), and set $W
\doteq\tilde{Y}$. Then $W$ is an $(\mathcal{F}_{t})$-Wiener process
under $\Prb$ and
\[
\Prb\circ\biggl( W + \int_{0}^{\cdot}u(s)\mrmd s
\biggr)^{-1} = \gamma\circ\biggl( \tilde{Y} + \int_{0}^{\cdot}u(s)\mrmd s
\biggr)^{-1} = \gamma\circ Y^{-1} = \gamma
\]
since $Y$ is the identity on $\mathcal{Y} = \Omega$. Finally, by
(\ref{EqLemmaREWienerEst}),
\[
R (\gamma\| \gamma_{0} ) \geq\Mean\biggl[ \frac{1}{2} \int
_{0}^{T} \bigl|u(s)\bigr|^{2}\mrmd s \biggr],
\]
where expectation is taken with respect to $\Prb= \gamma$.
\end{pf*}

%reA.1 #&#
\begin{rem}
Lemma \ref{LemmaREWiener} allows to derive a version of Theorem 3.1 in
Bou{\'e} and Dupuis \cite{bouedupuis98}, the representation theorem
for Laplace
functionals with respect to a Wiener process. The starting point here
as there is the following abstract representation formula for Laplace
functionals (Proposition 1.4.2 in Dupuis and Ellis \cite
{dupuisellis97}, page 27). Let
$\mathcal{S}$ be a Polish space, $\nu\in\mathcal{P}(\mathcal{S})$. Then
for all $f\dvtx\mathcal{S} \rightarrow\mathbb{R}$ bounded and
measurable,
%
%eA.8 #&#
\begin{equation}
\label{EqRELaplace} -\log\int_{\mathcal{S}} \RMe^{-f(x)} \nu(\mrmdd x)
= \inf_{\mu\in\mathcal
{P}(\mathcal{S})} \biggl\{ R ( \mu\| \nu) + \int
_{\mathcal
{S}} f(x)\mu(\mrmdd x) \biggr\}.
\end{equation}
With $\mathcal{S} = \mathcal{Y}$, $\nu= \gamma_{0}$ Wiener measure
as above, (\ref{EqRELaplace}) and Lemma \ref{LemmaREWiener}
imply that
\begin{eqnarray*}
&&-\log\int_{\mathcal{Y}} \RMe^{-f(y)} \gamma_{0}(\mrmdd y)
\\
&&\quad= \inf_{\gamma\in\mathcal{P}(\mathcal{Y})} \biggl\{ \inf_{u\in
\mathcal{U}: \Law(\bar{Y}^{u}) = \gamma} \Mean
\biggl[\frac{1}{2} \int_{0}^{T}
\bigl|u(t)\bigr|^{2}\mrmd t \biggr] + \int_{\mathcal{Y}} f(y)\gamma(\mrmdd y)
\biggr\}
\\
&&\quad= \inf_{\gamma\in\mathcal{P}(\mathcal{Y})} \inf_{u\in\mathcal{U}:
\Law(\bar{Y}^{u}) = \gamma} \Mean\biggl[
\frac{1}{2} \int_{0}^{T}
\bigl|u(t)\bigr|^{2}\mrmd t + f \bigl(\bar{Y}^{u} \bigr) \biggr]
\\
&&\quad= \inf_{u\in\mathcal{U}} \Mean\biggl[\frac{1}{2} \int
_{0}^{T} \bigl|u(t)\bigr|^{2}\mrmd t + f \bigl(
\bar{Y}^{u} \bigr) \biggr].
\end{eqnarray*}
Let $\hat{W}$ be a standard $d$-dimensional Wiener process over time
$[0,T]$ defined on some probability space $(\hat{\Omega},\hat
{\mathcal{F}},\hat{\Prb})$. Since $\int_{\mathcal{Y}} \RMe^{-f(y)}
\gamma_{0}(\mrmdd y) = \Mean_{\hat{\Prb}} [ \RMe^{-f(W)} ]$, it
follows that for all $f\dvtx\mathcal{S} \rightarrow\mathbb{R}$
bounded and measurable,
%
%eA.9 #&#
\begin{equation}
\label{EqRELaplaceWiener} -\log\Mean_{\hat{\Prb}} \bigl[ \RMe^{-f(W)}
\bigr] = \inf_{u\in
\mathcal{U}} \Mean\biggl[\frac{1}{2} \int
_{0}^{T} \bigl|u(t)\bigr|^{2}\mrmd t + f \bigl(
\bar{Y}^{u} \bigr) \biggr].
\end{equation}
The difference with the formula as stated in Bou{\'e} and Dupuis \cite
{bouedupuis98} lies in the fact that the control processes there all
live on the canonical space and are adapted to the canonical
filtration, while here the stochastic bases for the control processes
may vary; also cf. the related representation formula in Budhiraja and
Dupuis \cite{budhirajadupuis00}, where the control processes are
allowed to be adapted to filtrations larger than that induced by the
driving Wiener process.
\end{rem}

%sA #&#
\section{Sufficient conditions for hypotheses (H1)--(H5)} \label{AppItoSuff}

As in Section \ref{SectIto}, let $b$, $\sigma$ be predictable
functionals on $[0,T]\times\mathcal{X}\times\mathcal{P}(\mathbb
{R}^{d})$ with values in $\mathbb{R}^{d}$ and $\mathbb{R}^{d\times
d_{1}}$, respectively. Let $d_{bL}$ be the bounded Lipschitz metric on
$\mathcal{P}(\mathbb{R}^{d})$, that is,
\[
d_{bL}(\nu,\tilde{\nu})\doteq\sup\biggl\{ \int_{\mathbb{R}^{d}}
f(x)\nu(\mrmdd x) - \int_{\mathbb{R}^{d}} f(x)\tilde{\nu}(\mrmdd x)\dvt\|f\|
_{bL}\leq1 \biggr\},
\]
where $\|\cdot\|_{bL}$ is defined for functions $f\dvtx\mathbb
{R}^{d}\rightarrow\mathbb{R}$ by
\[
\|f\|_{bL}\doteq\sup_{x\in\mathbb{R}^{d}} \bigl|f(x)\bigr| + \sup
_{x,y\in
\mathbb{R}^{d}: x\neq y} \frac{|f(x)-f(y)|}{|x-y|}.
\]
If $X$, $Y$ are two $\mathbb{R}^{d}$-valued random variables defined
on the same probability space, then
\[
d_{bL} \bigl(\Law(X),\Law(Y) \bigr) \leq\Mean\bigl[ |X-Y| \bigr].
\]

Consider the following local Lipschitz and growth conditions on $b$,
$\sigma$.
\begin{longlist}[(G)]
\item[(L)] For every $M\in\mathbb{N}$ there exists $L_{M} > 0$ such
that for all $t\in[0,T]$, all $\phi,\tilde{\phi}\in\mathcal{X}$,
all $\nu,\tilde{\nu}\in\mathcal{P}(\mathbb{R}^{d})$,
\begin{eqnarray*}
&&\bigl\llvert b(t,\phi,\nu) - b(t,\tilde{\phi},\tilde{\nu})\bigr\rrvert
+ \bigl
\llvert\sigma(t,\phi,\nu) - \sigma(t,\tilde{\phi},\tilde{\nu})\bigr
\rrvert
\\
&&\quad\leq L_{M} \Bigl(\sup_{s\in[0,t]} \bigl|\phi(s)-\tilde{
\phi}(s)\bigr| + d_{bL}(\nu,\tilde{\nu}) \Bigr),
\end{eqnarray*}
whenever $\sup_{s\in[0,t]} |\phi(s)|\vee|\tilde{\phi}(s)| \leq M$.

\item[(G)] There exist a constant $K > 0$ such that for all $t\in
[0,T]$, all $\phi\in\mathcal{X}$, all $\nu\in\mathcal{P}(\mathbb
{R}^{d})$,
\[
\bigl|b(t,\phi,\nu)\bigr| \leq K \Bigl( 1 + \sup_{s\in[0,t]}\bigl|\phi(s)\bigr|
\Bigr),\qquad
\bigl|\sigma(t,\phi,\nu)\bigr| \leq K.
\]
\end{longlist}
The boundedness condition on $\sigma$ is used only in the verification
of hypothesis
(H4).

%prA.1 #&#
\begin{prop} \label{PropLimitPathwiseUnique}
Grant condition \textup{(L)}. Let $((\Omega,\mathcal{F},\Prb),(\mathcal
{F}_{t}),u,W)\in\mathcal{U}$. Suppose that $X$, $\tilde{X}$ are
solutions of (\ref{EqSDELimitControl}) over the time interval
$[0,T]$ under control $u$ with initial condition $X(0) = \tilde{X}(0)$
$\Prb$-almost surely. Then $X$, $\tilde{X}$ are indistinguishable,
that is,
\[
\Prb\bigl( X(t) = \tilde{X}(t) \mbox{ for all } t\in[0,T] \bigr) = 1.
\]
\end{prop}

\begin{pf}
For $M\in\mathbb{N,}$ define an $(\mathcal{F}_{t})$-stopping time
$\tau_M$ by
\[
\tau_{M}(\omega)\doteq\inf\biggl\{ t\in[0,T]\dvt\bigl|X(t,\omega)\bigr|\vee\bigl|
\tilde{X}(t,\omega)\bigr|\vee\int_{0}^{t} \bigl|u(s,
\omega)\bigr|^{2}\mrmd s \geq M \biggr\}
\]
with $\inf\varnothing= \infty$. Observe that $\Prb( \tau_{M} \leq T )
\to0$ as $M\to\infty$ since $X$, $\tilde{X}$ are continuous processes
and $\Mean[ \int_{0}^{T}|u(s)|^{2}\mrmd s ] < \infty$. Set
$\theta(t)\doteq\Law(X(t))$, $\tilde
{\theta}(t)\doteq\Law(\tilde{X}(t))$, $t\in[0,T]$. Using H{\"o}lder's
inequality, Doob's maximal inequality, the It{\^o} isometry, and
condition (L), we obtain for $M\in\mathbb{N}$, all $t\in[0,T]$,
\begin{eqnarray*}
&&\Mean\Bigl[\sup_{s\in[0,t]} \bigl\llvert X(s\wedge
\tau_{M}) - \tilde{X}(s\wedge\tau_{M})\bigr\rrvert
^2 \Bigr]
\\
&&\quad\leq4T \Mean\biggl[ \int_{0}^{t\wedge\tau_{M}}\bigl\llvert b
\bigl(r,X,\theta(r) \bigr) - b \bigl(r,\tilde{X},\tilde{\theta}(r) \bigr
)\bigr
\rrvert^{2} \mrmd r \biggr]
\\
&&\qquad{} + 4\Mean\biggl[ \int_{0}^{t\wedge\tau_{M}} \bigl\llvert
\sigma\bigl(r,X,\theta(r) \bigr) - \sigma\bigl(r,\tilde{X},\tilde
{\theta}(r)
\bigr) \bigr\rrvert^{2}\mrmd r \cdot\int_{0}^{t\wedge\tau
_{M}}\bigl|u(r)\bigr|^{2}\mrmd r
\biggr]
\\
&&\qquad{} + 16\Mean\biggl[ \int_{0}^{t\wedge\tau_{M}}\bigl\llvert
\sigma\bigl(r,X,\theta(r) \bigr) - \sigma\bigl(r,\tilde{X},\tilde
{\theta}(r)
\bigr)\bigr\rrvert^{2} \mrmd r \biggr]
\\
&&\quad\leq4T \Mean\biggl[ \int_{0}^{t\wedge\tau_{M}}\bigl\llvert b
\bigl(r,X,\theta(r) \bigr) - b \bigl(r,\tilde{X},\tilde{\theta}(r) \bigr
)\bigr
\rrvert^{2} \mrmd r \biggr]
\\
&&\qquad{} + (4M + 16 )\Mean\biggl[ \int_{0}^{t\wedge\tau
_{M}} \bigl
\llvert\sigma\bigl(r,X,\theta(r) \bigr) - \sigma\bigl(r,\tilde
{X},\tilde{
\theta}(r) \bigr) \bigr\rrvert^{2}\mrmd r \biggr]
\\
&&\quad\leq8L^{2}_{M} (T + M + 4 ) \Mean\biggl[ \int
_{0}^{t\wedge\tau_{M}} \Bigl(\sup_{s\in[0,r]}
\bigl|X(s)-\tilde{X}(s)\bigr|^{2} + d_{bL}\bigl(\theta(r),\tilde{
\theta}(r)\bigr)^{2} \Bigr) \mrmd r \biggr]
\\
&&\quad\leq16L^{2}_{M} (T + M + 4 ) \int_{0}^{t}
\Mean\Bigl[ \sup_{s\in[0,r]} \bigl\llvert X(s\wedge
\tau_{M})-\tilde{X}(s\wedge\tau_{M})\bigr\rrvert
^{2} \Bigr]\mrmd r.
\end{eqnarray*}
An application of Gronwall's lemma yields that
\[
\Mean\Bigl[\sup_{s\in[0,T]} \bigl\llvert X(s\wedge
\tau_{M}) - \tilde{X}(s\wedge\tau_{M})\bigr\rrvert
^2 \Bigr] = 0,
\]
hence $\Prb(X(t) = \tilde{X}(t)$ for all $t\leq\tau_{M}) =
1$ for all $M\in\mathbb{N}$. This implies the assertion since $\tau
_{M} \nearrow\infty$ as $M\to\infty$ $\Prb$-almost surely.
\end{pf}

Proposition \ref{PropLimitPathwiseUnique} says that under condition
(L) pathwise uniqueness holds for (\ref{EqSDELimitControl})
with respect to $\mathcal{U}$. As in the classical case of
uncontrolled It{\^o} diffusions, pathwise uniqueness implies uniqueness
in law. The proof of Proposition \ref{PropLimitWeakUnique} below is in
fact analogous to that of Proposition 1 in Yamada and Watanabe \cite
{yamadawatanabe71}; also cf. Proposition 5.3.20 in
Karatzas and Shreve \cite{karatzasshreve91}, page 309.

%prA.2 #&#
\begin{prop} \label{PropLimitWeakUnique}
Assume that pathwise uniqueness holds for (\ref{EqSDELimitControl})
given any deterministic initial condition. Let
$((\Omega,\mathcal{F},\Prb),(\mathcal{F}_{t}),u,W)\in\mathcal {U}$,
$((\tilde{\Omega},\tilde{\mathcal{F}},\tilde{\Prb}),(\tilde
{\mathcal{F}}_{t}),\tilde{u},\tilde{W})\in\mathcal{U}$ be such that
$\Prb\circ(u,W)^{-1} = \tilde{\Prb}\circ(\tilde{u},\tilde {W})^{-1}$ as
probability measures on $\mathcal{B}(\mathcal
{R}_{1}\times\mathcal{Y})$. Suppose that $X$, $\tilde{X}$ are solutions
of (\ref{EqSDELimitControl}) over the time interval $[0,T]$
under control $u$ and $\tilde{u}$, respectively, with initial condition
$x_{0}$ $\Prb$/$\tilde{\Prb}$-almost surely. Then $\Prb
\circ(X,u,W)^{-1} = \tilde{\Prb}\circ(\tilde{X},\tilde{u},\tilde
{W})^{-1}$ as probability measures on $\mathcal{B}(\mathcal{Z})$.
\end{prop}

\begin{pf*}{Proof (sketch)}
Set $\hat{\mathcal{Z}}\doteq\mathcal{X}\times\mathcal{X}\times
\mathcal{R}_{1}\times\mathcal{Y}$ and $\mathcal{G}\doteq\mathcal
{B}(\hat{\mathcal{Z}})$. Let $\hat{Z} = (Z, \tilde{Z}, \rho, \hat
{W})$ be the canonical process on $\hat{\mathcal{Z}}$, and let
$(\mathcal{G}_{t})_{t\in[0,T]}$ be the canonical filtration (i.e.,
the natural filtration of $Z$). Let $\mathbf{R}$ be the probability
measure on $\mathcal{B}(\mathcal{R}_{1}\times\mathcal{Y})$ given by
\[
\mathbf{R}\doteq\Prb\circ(u,W)^{-1} = \tilde{\Prb}\circ(\tilde{u},
\tilde{W})^{-1}.
\]
Let $\mathbf{Q}\dvt\mathcal{R}_{1}\times\mathcal{Y}\times\mathcal
{B}(\mathcal{X})$ be a regular conditional distribution of $\Law(X,u,W)$
given $(u,W)$; thus for all $A\in\mathcal{B}(\mathcal{R}_{1}\times
\mathcal{Y})$, all $B\in\mathcal{B}(\mathcal{X})$,
\[
\Prb\bigl(X\in B, (u,W)\in A \bigr) = \int_{A} \mathbf
{Q}(r,w;B) \mathbf{R} \bigl(\mathrm{d}(r,w) \bigr).
\]
Analogously, let $\tilde{\mathbf{Q}}\dvt\mathcal{R}_{1}\times
\mathcal{Y}\times\mathcal{B}(\mathcal{X})$ be a regular conditional
distribution of $\Law(\tilde{X},\tilde{u},\tilde{W})$ given
$(\tilde{u},\tilde{W})$. Define $\hat{\Prb}\in\mathcal{P}(\hat
{\mathcal{Z}})$ by setting, for $B,\tilde{B}\in\mathcal
{B}(\mathcal{X})$, $A\in\mathcal{B}(\mathcal{R}_{1}\times\mathcal{Y})$,
\[
\hat{\Prb} (B\times\tilde{B}\times A ) \doteq\int_{A}
\mathbf{Q}(r,w;B)\cdot\tilde{\mathbf{Q}}(r,w;\tilde{B}) \mathbf{R}
\bigl(\mathrm{d}(r,w)
\bigr).
\]
Let $\hat{\mathcal{G}}$ be the $\hat{\Prb}$-completion of $\mathcal
{G}$, and denote by $(\hat{\mathcal{G}}_{t})$ the right continuous
filtration induced by the $\hat{\Prb}$-augmentation of $(\mathcal
{G}_{t})$. Then $((\hat{\mathcal{Z}},\hat{\mathcal{G}},\hat{\Prb
}),(\hat{\mathcal{G}}_{t}),\rho,\hat{W}) \in\mathcal{U}$, where
$(\rho,\hat{W})$ are the last two components of the canonical process
$\hat{Z}$. One checks that
\[
\hat{\Prb}\circ(Z,\rho,\hat{W})^{-1} = \Prb\circ(X,u,W)^{-1},\qquad
\hat{\Prb}\circ(\tilde{Z},\rho,\hat{W})^{-1} = \tilde{\Prb}\circ(
\tilde{X},\tilde{u},\tilde{W})^{-1}
\]
and that $Z$, $\tilde{Z}$ are solutions of (\ref{EqSDELimitControl})
over the time interval $[0,T]$ under control
$((\hat{\mathcal{Z}},\hat{\mathcal{G}},\hat{\Prb}),(\hat
{\mathcal{G}}_{t}),\allowbreak\rho,\hat{W}) \in\mathcal{U}$ with initial condition
$x_{0}$ $\hat{\Prb}$-almost surely, where $\rho$ is being identified
with the control process $v(t)\doteq\int_{\mathbb {R}^{d_{1}}}y
\rho_{t}(\mrmdd y)$. By hypothesis, pathwise uniqueness holds for
(\ref{EqSDELimitControl}) with deterministic initial condition; it
follows that
\[
\hat{\Prb} \bigl( Z(t) = \tilde{Z}(t) \mbox{ for all }t\in[0,T] \bigr)
= 1,
\]
which implies $\Prb\circ(X,u,W)^{-1} = \tilde{\Prb}\circ(\tilde
{X},\tilde{u},\tilde{W})^{-1}$.
\end{pf*}

The following lemma is used in the verification of hypothesis (H4).

%leA.1 #&#
\begin{lemma} \label{LemmaModulusContinuity}
Let $(\Omega,\mathcal{F},\Prb),(\mathcal{F}_{t}))$ be a stochastic
basis satisfying the usual hypotheses, and let $M$ be a continuous
local martingale with respect to $(\mathcal{F}_{t})$ with quadratic
variation $\langle M\rangle$. Suppose there exists a finite constant
$C > 0$ such that for $\Prb$-almost all $\omega\in\Omega$, all
$t,s\in[0,T]$,
\[
\bigl\llvert\langle M\rangle(t,\omega) - \langle M\rangle(s,\omega)
\bigr
\rrvert\leq C\cdot|t-s|.
\]
Then for every $\delta_{0}\in(0,T]$,
\[
\Mean\Bigl[\sup_{\delta\in(0,\delta_{0}]} \delta^{-1/4}\cdot\sup
_{t,s\in[0,T]: |t-s|\leq\delta} \bigl\llvert M(t) - M(s)\bigr\rrvert
\Bigr] \leq192\cdot
\sqrt{C}\cdot(e\cdot T)^{1/4}.
\]
\end{lemma}

\begin{pf}
Since the assertion is about the behavior of $M$ only up to time $T$,
we may assume that $\lim_{t\to\infty} \langle M\rangle(t) =\infty$
$\Prb$-almost surely. For $s\geq0$ set $\tau_{s}\doteq\inf\{ t\geq
0\dvt\langle M\rangle(t) > s\}$. Then $\tau_{s}$ is an $(\mathcal
{F}_{t})$-stopping time for every $s\geq0$. By the
Dambis--Dubins--Schwarz theorem
(e.g., Theorem~3.4.6 in Karatzas and Shreve \cite{karatzasshreve91},
page 174), setting $W(t,\omega)\doteq M(\tau_{t}(\omega),\omega)$,
$t\geq0$, $\omega\in\Omega$, defines a standard Wiener process with
respect to the filtration $(\mathcal{F}_{\tau_{t}})$ and for $\Prb
$-almost all $\omega\in\Omega$, all $t\geq0$,
\[
M(t,\omega) = W \bigl(\langle M\rangle(t,\omega),\omega\bigr).
\]
Using the Garsia--Rodemich--Rumsey inequality one can show (cf. Appendix
in Fischer and Nappo \cite{fischernappo10}) that for every $p\geq1$,
every $\tilde{T} > 0$, there exists a $p$-integrable random variable
$\xi_{p,\tilde{T}}$ such that $\Mean[ |\xi_{p,\tilde
{T}}|^{p} ] \leq192^{p}\cdot p^{p/2}$ and for $\Prb$-almost
all $\omega\in\Omega$, all $t,s\in[0,T]$ such that $|t-s| \leq
\tilde{T}/e$,
\[
\bigl\llvert W(t,\omega) - W(s,\omega)\bigr\rrvert\leq\xi_{p,\tilde
{T}}(
\omega)\cdot\sqrt{|t-s|\log\biggl(\frac{\tilde
{T}}{|t-s|} \biggr)}.
\]
Clearly,\vspace*{1pt} $x\mapsto x\log(\tilde{T}/x)$ is increasing on $(0,\tilde
{T}/e]$, $\lim_{x\to0+} x\log(\tilde{T}/x) = 0$, and $(x\log
(\tilde{T}/\allowbreak x))^{1/2}\leq(\tilde{T}\cdot x)^{1/4}$ for all $x\in
(0,\tilde{T}/e]$. Since $\langle M\rangle$ is non-decreasing with
$\langle M\rangle(0) =0$ and, by hypothesis, $|\langle M\rangle(t) -
\langle M\rangle(s)| \leq C\cdot|t-s|$, it follows that for $\Prb
$-almost all $\omega\in\Omega$, every $\delta\in(0,T]$,
\begin{eqnarray*}
&&\sup_{t,s\in[0,T]: |t-s|\leq\delta} \bigl\llvert M(t,\omega) -
M(s,\omega)\bigr
\rrvert
\\
&&\quad= \sup_{t,s\in[0,T]: |t-s|\leq\delta} \bigl\llvert W \bigl(\langle
M\rangle(t,
\omega),\omega\bigr) - W \bigl(\langle M\rangle(s,\omega),\omega\bigr
)\bigr
\rrvert
\\
&&\quad\leq\sup_{t,s\in[0,T]: |t-s|\leq\delta} \xi_{p,e\cdot C\cdot
T}(\omega)
\\
&&\hspace*{55pt}\qquad{} \times\sqrt{\bigl|\langle M\rangle(t,\omega)-\langle M\rangle(s,\omega)\bigr|
\log
\biggl(\frac{e\cdot C\cdot T}{|\langle M\rangle
(t,\omega)-\langle M\rangle(s,\omega)|} \biggr)}
\\
&&\quad\leq\sup_{t,s\in[0,T]: |t-s|\leq\delta} \xi_{p,e\cdot C\cdot
T}(\omega)\cdot\sqrt{C}\cdot
\sqrt{\delta\log\biggl(\frac
{e\cdot T}{\delta} \biggr)}
\\
&&\quad\leq\sqrt{C}\cdot\xi_{p,e\cdot C\cdot T}(\omega)\cdot(e\cdot T\cdot
\delta
)^{1/4}.
\end{eqnarray*}
The assertion follows by choosing $p$ equal to one, inserting the term
containing the supremum over $\delta\in(0,\delta_{0}]$, and taking
expectations.
\end{pf}

%prA.3 #&#
\begin{prop}
Conditions \textup{(L)} and \textup{(G)} entail hypotheses \textup{(H1)--(H5)}.
\end{prop}

\begin{pf}
Hypothesis (H1)
is an immediate consequence of conditions (L) and (G). To verify
hypothesis \textup{(H2)}, let $N\in\mathbb{N}$ and define functions
$b_{N}\dvtx[0,T]\times\mathcal{X}^{N} \rightarrow\mathbb{R}^{N\times
d}$, $\sigma_{N}\dvtx[0,T]\times\mathcal{X}^{N} \rightarrow\mathbb
{R}^{N\times d\times N\times d_{1}}$ according to
\begin{eqnarray*}
b_{N}(t,\bolds{\phi})&\doteq&{ \bigl(b \bigl(t,\phi
_{1},\mu^{N}_{\bolds{\phi}(t)} \bigr),\ldots,b \bigl(t,\phi
_{N},\mu^{N}_{\bolds{\phi}(t)} \bigr) \bigr)}^\mathsf{T},
\\
\sigma_{N}(t,\bolds{\phi})&\doteq&\diag\bigl(\sigma\bigl(t,
\phi_{1},\mu^{N}_{\bolds{\phi}(t)} \bigr),\ldots,\sigma
\bigl(t,\phi_{N},\mu^{N}_{\bolds{\phi}(t)} \bigr) \bigr),
\end{eqnarray*}
where $\mu^{N}_{\bolds{\phi}(t)}\doteq
\frac{1}{N}\sum_{i=1}^{N}\delta_{\phi_{i}(t)}$. Then $b_{N}$,
$\sigma_{N}$ are the coefficients for the system of $N$ stochastic
differential equations given by (\ref{EqSDEPrelimit}). Thanks to
conditions (L) and (G), $b_{N}$, $\sigma_{N}$ are locally Lipschitz
continuous and of sub-linear growth. The It{\^o} existence and
uniqueness theorem (e.g., Theorem V.12.1
in Rogers and Williams \cite{rogerswilliams00b}, page 132) thus yields
pathwise uniqueness and
existence of strong solutions for the system of equations
(\ref{EqSDEPrelimit}). By Proposition \ref{PropLimitPathwiseUnique} in
conjunction with condition (L), pathwise uniqueness holds for
(\ref{EqSDELimitControl}). By
Proposition \ref{PropLimitWeakUnique}, it follows that weak uniqueness
holds for (\ref{EqSDELimitControl}); hence
hypothesis \textup{(H3)} is satisfied.

In order to verify hypothesis \textup{(H4)}, let $u^{N}\in\mathcal
{U}_{N}$, $N\in\mathbb{N}$, be such that
\[
\sup_{N\in\mathbb{N}} \frac{1}{N} \sum
_{i=1}^{N} \Mean\biggl[ \int_{0}^{T}\bigl|u_{i}^{N}(t)\bigr|^{2}\mrmd t
\biggr] <\infty.
\]
For $N\in\mathbb{N}$, let $\bar{\mu}^{N}$ be the empirical measure
of the solution to the system of equations (\ref
{EqSDEPrelimitControl}) under $u^{N}$. We have to show that $\{ \Prb
_{N}\circ(\bar{\mu}^{N})^{-1}\dvt N\in\mathbb{N}\}$ is tight in
$\mathcal{P}(\mathcal{P}(\mathcal{X}))$. Choose $\delta_{0}\in
(0,1\wedge
T]$, and define a function $G\dvtx\mathcal{P}(\mathcal{X}) \rightarrow
[0,\infty]$ by
\[
G(\theta)\doteq\int_{\mathcal{X}} \Bigl(\bigl|\phi(0)\bigr| + \sup
_{\delta
\in(0,\delta_{0}]} \delta^{-1/4}\cdot\sup_{t,s\in[0,T]:
|t-s|\leq\delta}
\bigl|\phi(t)-\phi(s)\bigr| \Bigr) \theta(\mrmdd \phi).
\]
Then $G$ is a tightness function, that is, $G$ is measurable and the
sublevel sets $\{\theta\dvtx G(\theta) \leq c\}$, $c\in[0,\infty)$, are
compact in $\mathcal{P}(\mathcal{X})$. This latter property is a consequence
of the Ascoli--Arzel{\`a} characterization of relatively compact sets in
$\mathcal{P}(\mathcal{X})$ (e.g., Theorem 2.4.9
in Karatzas and Shreve \cite{karatzasshreve91}, page 62), the Markov
inequality and Fatou's lemma.
We are going to show that $\sup_{N\in\mathbb{N}} \Mean_{N} [
G(\bar{\mu}^{N}) ] < \infty$, which implies that $\{
\Prb_{N}\circ(\bar{\mu}^{N})^{-1}\dvtx N\in\mathbb{N}\}$ is tight. By
construction, for $N\in\mathbb{N}$,
\begin{eqnarray*}
G\bigl(\bar{\mu}^{N}\bigr) &=& \frac{1}{N}\sum
_{i=1}^{N} \Bigl( \bigl\llvert\bar
{X}^{N}_{i}(0)\bigr\rrvert+ \sup_{\delta\in(0,\delta_{0}]}
\delta^{-1/4}\cdot\sup_{t,s\in[0,T]: |t-s|\leq\delta} \bigl\llvert
\bar
{X}^{N}_{i}(t) - \bar{X}^{N}_{i}(s)
\bigr\rrvert\Bigr)
\\
&=& |x_{0}| + \frac{1}{N}\sum_{i=1}^{N}
\sup_{\delta\in(0,\delta
_{0}]} \delta^{-1/4}\cdot\sup
_{t,s\in[0,T]: |t-s|\leq\delta} \biggl| \int_{s}^{t} b
\bigl(r,\bar{X}^{N}_{i},\bar{\mu}^{N}(r) \bigr)\mrmd r
\\
&&\hspace*{172pt}{} + \int_{s}^{t}\sigma\bigl(r,
\bar{X}^{N}_{i},\bar{\mu}^{N}(r)
\bigr)u^{N}_{i}(r)\mrmd r \\
&&\hspace*{172pt}{} + \int_{s}^{t}
\sigma\bigl(r,\bar{X}^{N}_{i},\bar{\mu}^{N}(r)
\bigr)\mrmd W^{N}_{i}(r) \biggr|.
\end{eqnarray*}
Thanks to condition (G), for every $i\in\{1,\ldots,N\}$,
\[
\sup_{\delta\in(0,\delta_{0}]} \delta^{-1/4}\cdot\sup
_{t,s\in
[0,T]: |t-s|\leq\delta} \biggl\llvert\int_{s}^{t}
b \bigl(r,\bar{X}^{N}_{i},\bar{\mu}^{N}(r)
\bigr)\mrmd r \biggr\rrvert\leq K \bigl(1 + \bigl\| X^{N}_{i}
\bigr\|_{\infty} \bigr)
\]
and, by H{\"o}lder's inequality,
\begin{eqnarray*}
&&\sup_{\delta\in(0,\delta_{0}]} \delta^{-1/4}\cdot\sup
_{t,s\in
[0,T]: |t-s|\leq\delta} \biggl\llvert\int_{s}^{t}
\sigma\bigl(r,\bar{X}^{N}_{i},\bar{\mu}^{N}(r)
\bigr)u^{N}_{i}(r)\mrmd r \biggr\rrvert
\\
&&\quad\leq\sqrt{T}K \cdot\sqrt{\int_{0}^{T}
\bigl|u^{N}_{i}(t)\bigr|^{2}\mrmd t} \leq\sqrt{T}K \biggl(
\frac{1}{2} + \frac{1}{2}\int_{0}^{T}
\bigl|u^{N}_{i}(t)\bigr|^{2}\mrmd t \biggr).
\end{eqnarray*}
The process $\int_{0}^{\cdot}\sigma(r,\bar{X}^{N}_{i},\bar{\mu
}^{N}(r))\mrmd W^{N}_{i}(r)$ is a vector of continuous local martingales
which, thanks to condition (G), satisfy the hypothesis of Lemma \ref
{LemmaModulusContinuity} with $C = K^{2}$. It follows that there exists
a finite constant $K_{T} > 0$ depending only on $K$ and $T$ such that
\begin{eqnarray*}
&&\Mean_{N} \bigl[G\bigl(\bar{\mu}^{N}\bigr) \bigr]
\\
&&\quad\leq|x_{0}| + K_{T} \Biggl(1 + \frac{1}{N}\sum
_{i=1}^{N}\Mean_{N} \bigl[
\bigl\|X^{N}_{i}\bigr\|_{\infty} \bigr] + \frac{1}{N}\sum
_{i=1}^{N} \Mean_{N} \biggl[\int
_{0}^{T} \bigl|u^{N}_{i}(t)\bigr|^{2}\mrmd t
\biggr] \Biggr).
\end{eqnarray*}
Since $\sup_{N\in\mathbb{N}} \frac{1}{N} \sum_{i=1}^{N} \Mean
[ \int_{0}^{T}|u_{i}^{N}(t)|^{2}\mrmd t ] <\infty$ by
hypothesis, it remains to check that, for some finite constant $\hat
{K}_{T} > 0$ depending only on $K$ and $T$,
\[
\frac{1}{N}\sum_{i=1}^{N}
\Mean_{N} \bigl[ \bigl\|X^{N}_{i}\bigr\|_{\infty}
\bigr] \leq\hat{K}_{T} \Biggl( 1 + \frac{1}{N}\sum
_{i=1}^{N} \Mean_{N} \biggl[\int
_{0}^{T} \bigl|u^{N}_{i}(t)\bigr|^{2}\mrmd t
\biggr] \Biggr).
\]
But this follows by standard arguments involving localization along the
stopping times $\tau^{N}_{M}\doteq\inf\{ t\in[0,T]\dvt\max_{i\in\{
1,\ldots,N\}} \sup_{s\leq t} |X^{N}_{i}(s)| \geq M\}$, $M\in\mathbb
{N}$, H{\"o}lder's inequality, Doob's maximal inequality, It{\^o}'s
isometry, condition (G), and Gronwall's lemma.

Hypothesis (H5) is again a consequence of the It{\^o} existence and
uniqueness theorem since under conditions (L) and (G), given any
$\theta\in\mathcal{P}(\mathcal{X})$, the mappings $(t,\phi)\mapsto
b(t,\phi,\theta(t))$, $(t,\phi)\mapsto\sigma(t,\phi,\theta(t))$
are predictable, locally Lipschitz continuous, and of sub-linear growth.
\end{pf}
\end{appendix}

\section*{Acknowledgements}

The author is grateful to Paolo Dai Pra for many stimulating questions
and discussions. The author thanks an anonymous Referee for her/his
helpful suggestions and critique. Partial financial support was
provided by the University of Padua through the Project ``Stochastic
Processes and Applications to Complex Systems'' (\mbox{CPDA123182}).

%suskaldyti doi

% imsref loaded by lrinkeviciute, 2013-08-13 13:34:28

\printhistory

\end{document}